\documentclass[final]{siamart1116}

\usepackage{lipsum}
\usepackage{amsfonts}
\usepackage{graphicx}
\usepackage{epstopdf}
\usepackage{algorithmic}
\ifpdf
  \DeclareGraphicsExtensions{.eps,.pdf,.png,.jpg}
\else
  \DeclareGraphicsExtensions{.eps}
\fi

\numberwithin{theorem}{section}

\newcommand{\TheTitle}{An Example Article} 
\newcommand{\TheAuthors}{D. Doe, P. T. Frank, and J. E. Smith}

\headers{\TheTitle}{\TheAuthors}

\title{Symplectic Model-Reduction with a Weighted Inner Product%
  \thanks{%
\funding{Babak Maboudi Afkham is supported by the SNSF under the grant number P1ELP2\_175039. Ashish Bhatt and Bernard Haasdonk gratefully acknowledge the support of DFG grant number HA5821/5-1.}} }

\author{Babak Maboudi Afkham%
  \thanks{Institute of Mathematics (MATH), School of Basic Sciences (FSB), Ecole Polytechnique F\'ed\'erale de Lausanne, 1015 Lausanne, Switzerland (\email{babak.maboudi@epfl.ch}, \email{jan.hesthaven@epfl.ch}).}%
  \and
  Ashish Bhatt%
  \thanks{University of Stuttgart, IANS, Pfaffenwaldring 57, 70569 Stuttgart, Germany (\email{[ashish.bhatt,haasdonk]@mathematik.uni-stuttgart.de}).}
  \and
  Bernard Haasdonk%
  \footnotemark[3]
  \and
  Jan S. Hesthaven%
  \footnotemark[2]
}

\headers{Symplectic Model-Reduction with a Weighted Inner Product}
{B. M. Afkham, A. Bhatt, B. Haasdonk, and J. S. Hesthaven}

\ifpdf
\hypersetup{ pdftitle={Guide to Using  SIAM'S \LaTeX\ Style} }
\fi

\begin{document}

\maketitle

% REQUIRED
\begin{abstract}
In the recent years, considerable attention has been paid to preserving structures and invariants in reduced basis methods, in order to enhance the stability and robustness of the reduced system. In the context of Hamiltonian systems, symplectic model reduction seeks to construct a reduced system that preserves the symplectic symmetry of Hamiltonian systems. However, symplectic methods are based on the standard Euclidean inner products and are not suitable for problems equipped with a more general inner product. In this paper we generalize symplectic model reduction to allow for the norms and inner products that are most appropriate to the problem while preserving the symplectic symmetry of the Hamiltonian systems. To construct a reduced basis and accelerate the evaluation of nonlinear terms, a greedy generation of a symplectic basis is proposed. Furthermore, it is shown that the greedy approach yields a norm bounded reduced basis. The accuracy and the stability of this model reduction technique is illustrated through the development of reduced models for a vibrating elastic beam and the sine-Gordon equation.
\end{abstract}

% REQUIRED
\begin{keywords}
Structure Preserving, Weighted MOR, Hamiltonian Systems, Greedy Reduced Basis, Symplectic DEIM
\end{keywords}
% REQUIRED
\begin{AMS}
78M34, 34C20, 35B30, 37K05, 65P10, 37J25
\end{AMS}

\section{Introduction}
\label{sec:intro}

Reduced order models have emerged as a powerful approach to cope with increasingly complex new applications in engineering and science. These methods substantially reduce the dimensionality of the problem by constructing a reduced configuration space. Exploration of the reduced space is then possible with significant acceleration \cite{hesthaven2015certified,Haasdonk2017}.

Over the past decade, reduced basis (RB) methods have demonstrated great success in lowering of the computational costs of solving elliptic and parabolic differential equations \cite{ito1998reduced,ito2001reduced}. However, model order reduction (MOR) of hyperbolic problems remains a challenge. Such problems often arise from a set of conservation laws and invariants. These intrinsic structures are lost during MOR which results in a qualitatively wrong, and sometimes unstable reduced system \cite{Amsallem:2014ef}.

%To have a sense of this error, error estimation is important from applications point of view \cite{HaasdonkOhlberger11,RuinerEtAl12,BhattEtAl18}. But it can difficult and expensive to compute useful error bounds. When one is interested in a cheap surrogate for the error incurred or when the conserved quantity is an output of the system, it becomes imperative to preserve this structure through model order reduction.

Recently, the construction of RB methods that conserve intrinsic structures has attracted attention \cite{doi:10.1137/17M1111991,1705.00498,kalashnikova2014stabilization,farhat2015structure,doi:10.1137/110836742,doi:10.1137/140959602,beattie2011structure,doi:10.1137/140978922}. Structure preservation in MOR not only constructs a physically meaningful reduced system, but can also enhance the robustness and stability of the reduced system. In system theory, conservation of passivity can be found in the work of \cite{polyuga2010structure,gugercin2012structure}. Energy preserving and inf-sup stable methods for finite element methods (FEM) are developed in \cite{farhat2015structure,ballarin2015supremizer}. Also, a conservative MOR technique for finite-volume methods is proposed in \cite{1711.11550}.

Moreover, the simulation of reduced models incurs solution errors and the estimation of this error is essential in applications of MOR \cite{HaasdonkOhlberger11,RuinerEtAl12,BhattEtAl18}. Finding tight error bounds for a general reduced system has shown to be computationally expensive and often impractical. Therefore, when one is interested in a cheap surrogate for the error or when the conserved quantity is an output of the system, it becomes imperative to preserve system structures in the reduced model.

In the context of Lagrangian and Hamiltonian systems, recent works provide a promising approach to the construction of robust and stable reduced systems. Carlberg, Tuminaro, and Boggs \cite{Carlberg:2014ky} suggest that a reduced order model of a Lagrangian system be identified by an approximate Lagrangian on a reduced order configuration space. This allows the reduced system to inherit the geometric structure of the original system. A similar approach has been adopted in the work of Peng and Mohseni \cite{doi:10.1137/140978922} and in the work of Maboudi Afkham and Hesthaven \cite{doi:10.1137/17M1111991} for Hamiltonian systems. They construct a low-order symplectic linear vector space, i.e. a vector space equipped with a symplectic 2-form, as the reduced space. Once the symplectic reduced space is generated, a symplectic projection result in a physically meaningful reduced system. A proper time-stepping scheme then preserves the Hamiltonian structure of the reduced system. It is shown in \cite{doi:10.1137/17M1111991,doi:10.1137/140978922} that this approach preserves the overall dynamics of the original system and enhances the stability of the reduced system. Despite the success of these method in MOR of Hamiltonian systems, these techniques are only compatible with the Euclidean inner product. Therefore, the computational structures that arise from a natural inner product of a problem will be lost during MOR.

Weak formulations and inner-products, defined on a Hilbert space, are at the core of the error analysis of many numerical methods for solving partial differential equations. Therefore, it is natural to seek MOR methods that consider such features. At the discrete level, these features often require a Euclidean vector space to be equipped with a generalized inner product, associated with a weight matrix $X$. Many works enabled conventional MOR techniques compatible with such inner products \cite{sen2006natural}. However, a MOR method that simultaneously preserves the symplectic symmetry of Hamiltonian systems remains unknown. 

In this paper, we seek to combine a classical MOR method with respect to a weight matrix with the symplectic MOR. The reduced system constructed by the new method is a generalized Hamiltonian system and the low order configuration space associated with this system is a symplectic linear vector space with a non-standard symplectic 2-form. It is demonstrated that the new method can be viewed as the natural extension to \cite{doi:10.1137/17M1111991}, and therefore retains the structure preserving features, e.g. symplecticity and stability. We also present a greedy approach for the construction of a generalized symplectic basis for the reduced system. Structured matrices are in general not norm bounded \cite{Karow:2006cf}. However, we show that the condition number of the basis generated by the greedy method is bounded by the condition number of the weight matrix $X$. Finally, to accelerate the evaluation of nonlinear terms in the reduced system, we present a variation of the discrete empirical interpolation method (DEIM) that preserves the symplectic structure of the reduced system.

What remains of this paper is organized as follows. In \cref{sec:hamil} we cover the required background on the Hamiltonian and the generalized Hamiltonian systems. \Cref{sec:mor} summarizes classic MOR routine with respect to a weighted norm and the symplectic MOR method with respect to the standard Euclidean inner product. We introduce the symplectic MOR method with respect to a weighted inner product in \cref{sec:normmor}. \Cref{sec:res} illustrates the performance of the new method through a vibrating beam and the sine-Gordon equation. We offer a few conclusive remarks in \cref{sec:conc}.

\section{Hamiltonian systems}
\label{sec:hamil}

In this section we discuss the basic concepts of the geometry of symplectic linear vector spaces and introduce Hamiltonian and Generalized Hamiltonian systems.

\subsection{Generalized Hamiltonian systems}
\label{sec:hamil.1}

Let $(\mathbb R^{2n}, \Omega)$ be a symplectic linear vector space, with $\mathbb R^{2n}$ the configuration space and $\Omega:\mathbb R^{2n}\times\mathbb R^{2n} \to \mathbb R$ a closed, skew-symmetric and non-degenerate 2-form on $\mathbb R^{2n}$. Given a smooth function $H:\mathbb R^{2n} \to \mathbb R$, the so called \emph{Hamiltonian}, the \emph{generalized Hamiltonian system} of evolution reads
\begin{equation} \label{eq:hamil.1}
\left\{
\begin{aligned}
	& \dot z = J_{2n} \nabla_z H(z),  \\
	&  z(0) = z_0.
\end{aligned}
\right.
\end{equation}
Here $z\in \mathbb R^{2n}$ are the configuration coordinates and $J_{2n}$ is a constant, full-rank and skew-symmetric $2n\times 2n$ \emph{structure} matrix such that $\Omega(x,y) = x^TJ_{2n}y$, for all state vectors $x,y\in \mathbb R^{2n}$ \cite{Marsden:2010:IMS:1965128}. Note that there always exists a coordinate transformation $\tilde z = \mathcal T^{-1} z$, with $\mathcal T \in \mathbb R^{2n\times 2n}$, such that $J_{2n}$ takes the form of the \emph{standard} symplectic structure matrix
\begin{equation} \label{eq:hamil.2}
	\mathbb{J}_{2n} = 
	\begin{pmatrix}
	0_n & I_n \\
	-I_n & 0_n
	\end{pmatrix},
\end{equation}
in the new coordinate system \cite{de2006symplectic}.
Here $0_n$ and $I_n$ are the zero matrix and the identity matrix of size $n\times n$, respectively. A central feature of Hamiltonian systems is conservation of the Hamiltonian.

\begin{theorem} \label{thm:1}
\cite{Marsden:2010:IMS:1965128} The Hamiltonian $H$ is a conserved quantity of the Hamiltonian system \eqref{eq:hamil.1} i.e. $H(z(t)) = H(z_0)$ for all $t \geq 0$.
\end{theorem}

Under a general coordinate transformation, the equations of evolution of a Hamiltonian system might not take the form (\ref{eq:hamil.1}). Indeed only transformations which preserve the symplectic form, \emph{symplectic transformations}, preserve the form of a Hamiltonian system \cite{Hairer:1250576}. Suppose that $(\mathbb R^{2n},\Omega)$ and $(\mathbb R^{2k},\Lambda)$ are two symplectic linear vector spaces. A transformation $\mu:\mathbb R^{2n}\to\mathbb R^{2k}$ is a symplectic transformation if
\begin{equation} \label{eq:hamil.3}
	\Omega(x,y) = \Lambda(\mu(x),\mu(y)), \quad \text{for all } x,y\in\mathbb R^{2n}.
\end{equation}
In matrix notation, i.e. when we consider a set of basis vectors for $\mathbb R^{2n}$ and $\mathbb R^{2k}$, a linear symplectic transformation is of the form $\mu(x) = Ax$ with $A\in \mathbb R^{2n\times 2k}$ such that
\begin{equation} \label{eq:hamil.4}
	A^T J_{2n} A = J_{2k}.
\end{equation}
We are interested in a class of symplectic transformations that transform a symplectic structure $J_{2n}$ into the standard symplectic structure $\mathbb J_{2k}$.
\begin{definition} \label{def:symp-mat}
Let $J_{2n}\in \mathbb R^{2n\times 2n}$ be a full-rank skew-symmetric structure matrix. A matrix $A\in\mathbb R^{2n\times 2k}$ is $J_{2n}$-symplectic if
\begin{equation} \label{eq:hamil.5}
A^T J_{2n} A = \mathbb{J}_{2k}.
\end{equation}
\end{definition}
Note that in the literature \cite{Marsden:2010:IMS:1965128,Hairer:1250576}, symplectic transformations refer to $\mathbb{J}_{2n}$-symplectic matrices, in contrast to \Cref{def:symp-mat}.

It is natural to expect a numerical integrator that solves (\ref{eq:hamil.1}) to also satisfy the conservation law expressed in  \Cref{thm:1}. Conventional numerical time integrators, e.g. general Runge-Kutta methods, do not generally preserve the symplectic symmetry of Hamiltonian systems which often result in an unphysical behavior of the solution over long time-integration. \emph{Poisson integrators} \cite{Hairer:1250576} are known to preserve the Hamiltonian of \eqref{eq:hamil.1}. To construct a general Poisson integrator, we seek a coordinate transformation $\mathcal T:\mathbb R^{2n}\to\mathbb R^{2n}$, $\tilde z = \mathcal T^{-1}z$, such that $J_{2n} = \mathcal T \mathbb J_{2n} \mathcal T^T$. Then, a \emph{symplectic integrator} can preserve the symplectic structure of the transformed system. The \emph{St\"ormer-Verlet} scheme is an example of a second order symplectic time-integrator given as
\begin{equation} \label{eq:hamil.6}
	\begin{aligned}
	q_{m+1/2} &= q_m + \frac{\Delta t} 2 \cdot \nabla_p \tilde H(p_m,q_{m+1/2}), \\
	p_{m+1} &= p_m - \frac{\Delta t} 2  \cdot \left( \nabla_q \tilde H(p_m,q_{m+1/2}) + \nabla_{q} \tilde H(p_{m+1},q_{m+1/2}) \right), \\
	q_{m+1} &= q_{m+1/2} + \frac{\Delta t} 2  \cdot  \nabla_p \tilde H(p_{m+1},q_{m+1/2}).
	\end{aligned}
\end{equation}
Here, $\tilde z = (q^T,p^T)^T$, $\tilde H(\tilde z) = H(\mathcal T^{-1}z)$, $\Delta t$ denotes a uniform time step-size, and $q_m \approx q(m\Delta t)$ and $p_m \approx p(m\Delta t)$, $m \in \mathbb{N} \cup \{ 0\}$, are approximate numerical solutions. Note that it is important to use a backward stable method to compute the transformation $\mathcal T$. In this paper we use the symplectic Gaussian elimination method with complete pivoting to compute the decomposition $J_{2n} = \mathcal T \mathbb J_{2n} \mathcal T^T$. However, one may use a more computationally efficient method, e.g., a Cholesky-like factorization proposed in \cite{benner:chol} or the isotropic Arnoldi/Lanczos methods \cite{doi:10.1137/S1064827500366434}. There are a few known numerical integrators that preserve the symplectic symmetry of a generalized Hamiltonian system without requiring the computation of the transformation matrix $\mathcal T$ \cite{Hairer:1250576}. The implicit midpoint rule
\begin{equation} \label{eq:hamil.7}
	z_{m+1} = z_{m} + \Delta t \cdot J_{2n} \nabla_z H \left( \frac{z_{m+1} + z_m}{2} \right),
\end{equation}
for \cref{eq:hamil.1} is an example of such integrators. For more on the construction and the applications of Poisson/symplectic integrators, we refer the reader to \cite{Hairer:1250576,bhatt2017structure}.

\section{Model order reduction}
\label{sec:mor}

In this section we summarize the fundamentals of MOR and discuss the conventional approach to MOR with a weighted inner product. We then recall the main results from \cite{doi:10.1137/17M1111991} regarding symplectic MOR. In \cref{sec:normmor} we shall combine the two concepts to introduce the symplectic MOR of Hamiltonian systems with respect to a weighted inner product.

\subsection{Model-reduction with a weighted inner product} \label{sec:mor.1}
Consider a dynamical system of the form
\begin{equation} \label{eq:mor.1}
\left\{
\begin{aligned}
	\dot x(t) &= f(t,x), \\
	x(0) &= x_0.
\end{aligned}
\right.
\end{equation}
where $x\in \mathbb R^{m}$ and $f:\mathbb R \times \mathbb R^{m} \to \mathbb R^{m}$ is some continuous function. In this paper we assume that the time $t$ is the only parameter on which the solution vector $x$ depends. Nevertheless, it is straightforward to generalize the findings of this paper to the case of parametric MOR, where $x$ depends on a larger set of parameters that belong to a closed and bounded subset.

Suppose that $x$ is well approximated by a low dimensional linear subspace with the basis matrix $V=[v_1|\dots|v_k]\in \mathbb R^{m\times k}$, $v_i\in \mathbb R^{m}$ for $i=1,\dots,k$. The approximate solution to (\ref{eq:mor.1}) in this basis reads
\begin{equation} \label{eq:mor.2}
	x \approx Vy,
\end{equation}
where $y \in \mathbb R^k$ are the expansion coefficients of $x$ in the basis $V$. Note that projection of $x$ onto colspan$(V)$ depends on the inner product and the norm defined on (\ref{eq:mor.1}). We define the weighted inner product
\begin{equation} \label{eq:mor.3}
	\left\langle x,y \right\rangle_X = x^TXy,\quad \text{for all } x,y \in \mathbb R^m,
\end{equation}
for some symmetric and positive-definite matrix $X\in \mathbb{R}^{m\times m}$ and refer to $\|\cdot \|_X$ as the $X$-norm associated to this inner product. If we choose $V$ to be an orthonormal basis with respect to the $X$-norm ($V^TXV=I_k$), then the operator
\begin{equation} \label{eq:mor.4}
	P_{X,V}(x) = VV^TXx, \quad \text{for all } x\in \mathbb R^{m}
\end{equation}
becomes idempotent, i.e. $P_{X,V}$ is a projection operator onto colspan$(V)$.

Now suppose that the \emph{snapshot matrix} $S=[x(t_1)|x(t_2)|\ldots|x(t_N)]$ is a collection of $N$ solutions to (\ref{eq:mor.1}) at time instances $t_1,\dots,t_N$. We seek $V$ such that it minimizes the collective projection error of the samples onto colspan$(V)$ which corresponds to the minimization problem
\begin{equation} \label{eq:mor.5}
\begin{aligned}
& \underset{V\in \mathbb{R}^{m\times k}}{\text{minimize}}
& & \sum_{i=1}^N \| x(t_i) - P_{X,V}( x(t_i) ) \|_X^2, \\
& \text{subject to}
& & V^TXV = I_k.
\end{aligned}
\end{equation}
Note that the solution to (\ref{eq:mor.5}) is known as the proper orthogonal decomposition (POD) \cite{hesthaven2015certified,quarteroni2015reduced,gubisch2017proper}. Following \cite{quarteroni2015reduced} the above minimization is equivalent to
\begin{equation} \label{eq:mor.6}
\begin{aligned}
& \underset{\tilde V\in \mathbb{R}^{m\times k}}{\text{minimize}}
& & \| \tilde S - \tilde V \tilde V^T \tilde S \|_F^2, \\
& \text{subject to}
& & \tilde V^T\tilde V = I_k.
\end{aligned}
\end{equation}
where $\tilde V = X^{1/2} V$, $\tilde S = X^{1/2} S$, and $X^{1/2}$ is the matrix square root of $X$. According to the Schmidt-Mirsky-Eckart-Young theorem \cite{Markovsky:2011:LRA:2103589} the solution $\tilde V$ to the minimization (\ref{eq:mor.6}) is the truncated singular value decomposition (SVD) of $\tilde S$. The basis $V$ then is $V = X^{-1/2}\tilde V$. The reduced model of (\ref{eq:mor.1}), using the basis $V$ and the projection $P_{X,V}$, is
\begin{equation} \label{eq:mor.7}
	\left\{
	\begin{aligned}
	\dot y(t) &= V^TX f(t,Vy), \\
	y(0) &= V^TX x_0.
	\end{aligned}
	\right.
\end{equation}
If $k$ can be chosen such that $k \ll m$, then the reduced system (\ref{eq:mor.7}) can potentially be evaluated significantly faster than the full order system (\ref{eq:mor.1}). Finding the matrix square root of $X$ can often be computationally exhaustive. In such cases, explicit use of $X^{1/2}$ can be avoided by finding the eigen-decomposition of the \emph{Gramian} matrix $G = S^TXS$ \cite{quarteroni2015reduced,Haasdonk2017}.

Besides RB methods, there exist other ways of basis generation e.g. greedy strategies, the Krylov subspace method, balanced truncation, Hankel-norm approximation etc. \cite{antoulas2005approximation}. We refer the reader to \cite{hesthaven2015certified,quarteroni2015reduced,Haasdonk2017} for further information regarding the development and the efficiency of reduced order models. 

\subsection{Symplectic MOR} \label{sec:mor.2}
Conventional MOR methods, e.g. those introduced in \cref{sec:mor.1}, do no generally preserve the conservation law expressed in \cref{thm:1}. As mentioned earlier, this often results in the lack of robustness in the reduced system over long time-integration. In this section we summarize the main findings of \cite{doi:10.1137/17M1111991} regarding symplectic model order reduction of Hamiltonian systems with respect to the standard Euclidean inner product. Symplectic MOR aims to construct a reduced system that conserves the geometric symmetry expressed in \Cref{thm:1} which helps with the stability of the reduced system.
Consider a Hamiltonian system of the form
\begin{equation} \label{eq:mor.8}
\left\{
\begin{aligned}
	\dot z(t) &= \mathbb J_{2n} L z(t) + \mathbb J_{2n} \nabla_z f(z), \\
	z(0) &= z_0.
\end{aligned}
\right.
\end{equation}
Here $z\in \mathbb R^{2n}$ is the state vector, $L\in\mathbb R^{2n\times 2n}$ is a symmetric and positive-definite matrix and $f:\mathbb R^{2n}\to\mathbb R$ is sufficiently smooth function. Note that the Hamiltonian for system (\ref{eq:mor.8}) is given by $H(z) = \frac 1 2 z^TLz + f(z)$. Suppose that the solution to (\ref{eq:mor.8}) is well approximated by a low dimensional symplectic subspace. Let $A\in \mathbb{R}^{2n\times 2k}$ be a $\mathbb{J}_{2n}$-symplectic basis containing the basis vectors $A=[e_1|\dots|e_k|f_1|\dots|f_k]$, such that $z \approx Ay$ with $y \in \mathbb{R}^{2k}$ the expansion coefficients of $z$ in this basis. Using the symplectic inverse $A^+ := \mathbb J_{2k}^T A^T \mathbb J_{2n}$ we can construct the reduced system
\begin{equation} \label{eq:mor.9}
	\dot y = A^+ \mathbb J_{2n} (A^+)^T A^T L A y + A^+ \mathbb J_{2n} (A^+)^T \nabla_y f(Ay).
\end{equation}
We refer the reader to \cite{doi:10.1137/17M1111991} for the details of the derivation. It is shown in \cite{doi:10.1137/140978922} that $(A^+)^T$ is also $\mathbb J_{2n}$-symplectic, therefore $A^+ \mathbb J_{2n} (A^+)^T = \mathbb J_{2k}$ and (\ref{eq:mor.9}) reduces to
\begin{equation} \label{eq:mor.10}
	\dot y(t) = \mathbb J_{2k} A^T L A y + \mathbb J_{2k} \nabla_y f(Ay).
\end{equation}
This system is a Hamiltonian system with the Hamiltonian $\mathcal H(y) = \frac 1 2 y^T A^T L A y + f(Ay)$. To reduce the complexity of evaluating the nonlinear term in (\ref{eq:mor.10}), we may apply the discrete empirical interpolation method (DEIM) \cite{barrault2004empirical,Chaturantabut:2010cz,wirtz2014posteriori}. Assuming that $\nabla_z f(z)$ lies near a low dimensional subspace with a basis matrix $U\in \mathbb R^{2n\times r}$ the DEIM approximation reads
\begin{equation} \label{eq:mor.11}
	\nabla_z f(z) \approx U (\mathcal P^T U)^{-1} \mathcal P^T \nabla_z f(z).
\end{equation}
Here $\mathcal P \in \mathbb R^{2n\times r}$ is the interpolating index matrix \cite{Chaturantabut:2010cz}. For a general choice of $U$ the approximation in (\ref{eq:mor.11}) destroys the Hamiltonian structure, if inserted in (\ref{eq:mor.8}). It is shown in \cite{doi:10.1137/17M1111991} that by taking $U = (A^+)^T$ we can recover the Hamiltonian structure in (\ref{eq:mor.10}). Therefore, the reduced system to (\ref{eq:mor.8}) becomes
\begin{equation} \label{eq:mor.12}
\left\{
\begin{aligned}
	\dot y(t) &= \mathbb J_{2k} A^T L A y + \mathbb J_{2k} (A^+)^T(\mathcal P^T (A^+)^T)^{-1} \mathcal P^T \nabla_z f(Ay), \\
	y(0) &= A^+ z_0.
\end{aligned}
\right.
\end{equation}
Note that the Hamiltonian formulation of (\ref{eq:mor.12}) allows us to integrate it using a symplectic integrator. This conserves the symmetry expressed in \Cref{thm:1} at the level of the reduced system. It is also shown in \cite{doi:10.1137/17M1111991,doi:10.1137/140978922} that the stability of the critical points of (\ref{eq:mor.8}) is preserved in the reduced system and the difference of the Hamiltonians of the two system \cref{eq:mor.8,eq:mor.12} is constant. Therefore, the overall behavior (\ref{eq:mor.12}) is close to the full order Hamiltonian system (\ref{eq:mor.8}). In the next subsection we discuss methods for generating a $\mathbb J_{2n}$-symplectic basis $A$.

\subsection{Greedy generation of a $\mathbb J_{2n}$-symplectic basis} \label{sec:mor.3}
Suppose that $S \in \mathbb R^{2n\times N}$ is the snapshot matrix containing the time instances $\{z(t_i)\}_{i=1}^N$ of the solution to (\ref{eq:mor.8}). We seek the $\mathbb J_{2n}$-symplectic basis $A$ such that the collective symplectic projection error of samples in $S$ onto colspan$(A)$ is minimized.
\begin{equation} \label{eq:mor.13}
\begin{aligned}
& \underset{A\in \mathbb{R}^{2n\times 2k}}{\text{minimize}}
& & \| S - P^\text{symp}_{I,A}(S) \|_F^2, \\
& \text{subject to}
& & A^T\mathbb J_{2n}A = \mathbb J_{2k}.
\end{aligned}
\end{equation}
Here $P^\text{symp}_{I,A} = AA^+$ is the symplectic projection operator with respect to the standard Euclidean inner product onto colspan$(A)$. Note that $P^\text{symp}_{I,A} \circ P^\text{symp}_{I,A} = P^\text{symp}_{I,A}$ \cite{doi:10.1137/140978922,doi:10.1137/17M1111991}.

Direct approaches to solve (\ref{eq:mor.13}) are often inefficient. Some SVD-type solutions to (\ref{eq:mor.13}) are proposed by \cite{doi:10.1137/140978922}. However, the form of the suggested basis, e.g. the block diagonal form suggested in \cite{doi:10.1137/140978922}, is not compatible with a general weight matrix $X$. 

The greedy generation of a $\mathbb J_{2n}$-symplectic basis aims to find a near optimal solution to (\ref{eq:mor.13}) in an iterative process. This method increases the overall accuracy of the basis by adding the best possible basis vectors at each iteration. Suppose that $A_{2k} = [e_1|\dots|e_k|\mathbb J_{2n}^T e_1|\dots|\mathbb J_{2n}^T e_k]$ is a $\mathbb J_{2n}$-symplectic and orthonormal basis \cite{doi:10.1137/17M1111991}. The first step of the greedy method is to find the snapshot $z_{k+1}$, that is worst approximated by the basis $A_{2k}$:
\begin{equation} \label{eq:mor.14}
	z_{k+1} := \underset{z \in \{ z(t_i) \}_{i=1}^N}{\text{argmax } }\| z - P^\text{symp}_{I,A_{2k}}(z) \|_2. 
\end{equation}
Note that if $z_{k+1}\neq 0$ then $z_{k+1}$ is not in colspan$(A_{2k})$. Then we obtain a non-trivial vector $e_{k+1}$ by $\mathbb J_{2n}$-orthogonalizing $z_{k+1}$ with respect to $A_{2k}$:
\begin{equation} \label{eq:mor.14.1}
	\tilde z = z_{k+1} - A_{2k}\alpha, \quad e_{k+1} = \frac{\tilde z}{\|\tilde z \|_2}.
\end{equation}
Here, $\alpha\in \mathbb R^{2k}$ are the expansion coefficients of the projection of $z$ onto the column span of $A_{2k}$ where $\alpha_i = -\Omega(z_{k+1},\mathbb J_{2n}^Te_i)$ for $i\leq k$ and $\alpha_i = \Omega(z_{k+1},e_i)$ for $i>k$. Since $\Omega(e_{k+1},\mathbb{J}_{2n}^T e_{k+1}) = \| e_{k+1} \|_2^2 \neq 0$ the enriched basis $A_{2k+2}$ reads
\begin{equation} \label{eq:mor.15}
	A_{2k+2} = [e_1|\dots|e_k|e_{k+1}|\mathbb J_{2n}^Te_1|\dots|\mathbb J_{2n}^Te_{k+1}].	
\end{equation}
It is easily verified that $A_{2k+2}$ is $\mathbb J_{2n}$-symplectic and orthonormal. This enrichment continues until the given tolerance is satisfied. We note that the choice of the orthogonalization routine generally depends on the application. In this paper we use the symplectic Gram-Schmidt (GS) process as the orthogonalization routine. However the isotropic Arnoldi method or the isotropic Lanczos method \cite{doi:10.1137/S1064827500366434} are backward stable alternatives.

MOR is specially useful in reducing parametric models that depend on a closed and bounded parameter set $\mathcal{S} \subset \mathbb R^{d}$ characterizing physical properties of the underlying system. The evaluation of the projection error is impractical for such problems. The loss in the Hamiltonian function can be used as a cheap surrogate to the projection error. Suppose that a $\mathbb J_{2n}$-symplectic basis $A_{2k}$ is given, then one selects a new parameter $\omega_{k+1} \in \mathcal{S}$ by greedy approach:
\begin{equation} \label{eq:mor.16}
	\omega_{k+1} = \underset{\omega \in \mathcal{S}}{\text{argmax } } | H(z(\omega)) - H(P^\text{symp}_{I,A}(z(\omega))) |,
\end{equation}
and then enriches the basis $A_{2k}$ as discussed above. It is shown in \cite{doi:10.1137/17M1111991} that the loss in the Hamiltonian is constant in time. Therefore, $\omega_{k+1}$ can be identified in the \emph{offline phase} before simulating the reduced order model. Note that the relation between the projection error \cref{eq:mor.14} and the error in the Hamiltonian \cref{eq:mor.16} is still unknown.

We summarize the greedy algorithm for generating a $\mathbb J_{2n}$-symplectic basis in \Cref{alg:1}. The first loop constructs a $\mathbb J_{2n}$-symplectic basis for the Hamiltonian system (\ref{eq:mor.8}), and the second loop adds the nonlinear snapshots to the symplectic inverse of the basis. We refer the reader to \cite{doi:10.1137/17M1111991} for more details. In \cref{sec:normmor} we will show how this algorithm can be generalized to support any weighted inner product.

\begin{algorithm} 
\caption{The greedy algorithm for generation of a $\mathbb J_{2n}$-symplectic basis} \label{alg:1}
{\bf Input:} Tolerated projection error $\delta$, initial condition $ z_0$, snapshots $\mathcal Z = \{ z(t_i) \}_{i=1}^{N}$ and $\mathcal G = \{ \nabla f(z(t_i)) \}_{i=1}^{N}$
\begin{enumerate}
%\item $t^1 \leftarrow t=0$
\item $e_1 \leftarrow \frac{z_0}{\|z_0\|_2}$
\item $A \leftarrow [e_1|\mathbb J^T_{2n}e_1]$
\item $k \leftarrow 1$
\item \textbf{while} $\| z - P^\text{symp}_{I,A}( z ) \|_2 > \delta$ for any $z\in \mathcal Z$
\item \hspace{0.5cm} $z_{k+1} := \underset{z\in \mathcal Z}{\text{argmax }} \| z - P^\text{symp}_{I,A}( z ) \|_2$
\item \hspace{0.5cm} $\mathbb J_{2n}$-orthogonalize $ z_{k+1}$ to obtain $e_{k+1}$
\item \hspace{0.5cm} $A \leftarrow [e_1|\dots |e_{k+1} | \mathbb J^T_{2n}e_1|\dots,\mathbb J^T_{2n}e_{k+1}]$
\item \hspace{0.5cm} $k \leftarrow k+1$
\item \textbf{end while}
\item compute $(A^+)^T=[e'_1|\dots|e'_k|\mathbb J^T_{2n}e'_1|\dots|\mathbb J^T_{2n}e'_k]$
\item \textbf{while} $\| g - P^\text{symp}_{I,(A^+)^T}(g) \|_2 > \delta$ for all $g \in \mathcal G$
\item \hspace{0.5cm} $g_{k+1} := \underset{g \in \mathcal G}{\text{argmax }} \| g - P^\text{symp}_{I,(A^+)^T}(g) \|_2$
\item \hspace{0.5cm} $\mathbb J_{2n}$-orthogonalize $g_{k+1}$ to obtain $e'_{k+1}$
\item \hspace{0.5cm} $(A^+)^T \leftarrow [e'_1|\dots |e'_{k+1} | \mathbb J^T_{2n}e'_1|\dots|\mathbb J^T_{2n}e'_{k+1}]$
\item \hspace{0.5cm} $k \leftarrow k+1$
\item \textbf{end while}
\item $A \leftarrow \left( \left( \left( A^+\right) ^T \right ) ^+ \right)^T$
\end{enumerate}
\vspace{0.5cm}
{\bf Output:} $\mathbb J_{2n}$-symplectic basis $A$.
\end{algorithm}

\section{Symplectic MOR with weighted inner product} \label{sec:normmor}

In this section we combine the concept of model reduction with a weighted inner product, discussed in \cref{sec:mor.1}, with the symplectic model reduction discussed in \cref{sec:mor.2}. We will argue that the new method can be viewed as a natural extension of the original symplectic method. Finally, we generalize the greedy method for the symplectic basis generation, and the symplectic model reduction of nonlinear terms to be compatible with any non-degenerate weighted inner product.

\subsection{Generalization of the symplectic projection} \label{sec:normmor.1}
As discussed in \cref{sec:mor.1}, the error analysis of methods for solving partial differential equations often requires the use of a weighted inner product. This is particularly important when dealing with Hamiltonian systems, where the system energy can induce a norm that is fundamental to the dynamics of the system.

Consider a Hamiltonian system of the form (\ref{eq:mor.8}) together with the weighted inner product defined in (\ref{eq:mor.3}) with $m=2n$. Also suppose that the solution $z$ to (\ref{eq:mor.8}) is well approximated by a $2k$ dimensional symplectic subspace with the basis matrix $A$. We seek to construct a projection operator that minimizes the projection error with respect to the $X$-norm while preserving the symplectic dynamics of (\ref{eq:mor.8}) in the projected space. Consider the operator $P: \mathbb R^{2n} \to \mathbb R^{2n}$ be defined as
%\begin{align*}
%\| X^{1/2} S - X^{1/2} A X^{1/2} A ^+ X^{1/2} S \|_F^2 = \| X^{1/2} S - X^{1/2} A \mathbb J_{2k}^T A^T X \mathbb J_{2n} X S \|_F^2.
%\end{align*}
\begin{equation} \label{eq:normmor.1}
	P = A \mathbb J_{2k}^T A^T X \mathbb J_{2n} X.
\end{equation}
It is easy to show that $P$ is idempotent if and only if
\begin{equation} \label{eq:normmor.2}
	\mathbb J_{2k}^T A^T X \mathbb J_{2n} X A = I_{2k},
\end{equation}
in which case $P$ is a projection operator onto colspan$(A)$. Suppose that $S$ is the snapshot matrix containing the time samples $\{z(t_i)\}_{i=1}^N$ of the solution to (\ref{eq:mor.8}). We seek to find the basis $A$ that minimizes the collective projection error of snapshots with respect to the $X$-norm,
\begin{equation} \label{eq:normmor.3}
\begin{aligned}
& \underset{A\in \mathbb{R}^{2n\times 2k}}{\text{minimize}}
& & \sum_{i=1}^N \| z(t_i) - P(z(t_i)) \|_X^2, \\
& \text{subject to}
& & \mathbb J_{2k}^T A^T X \mathbb J_{2n} X A = I_{2k}.
\end{aligned}
\end{equation}
By (\ref{eq:normmor.1}) we have
\begin{equation} \label{eq:normmor.4}
\begin{aligned}
	\sum_{i=1}^N \| z(t_i) - P(z(t_i)) \|_X^2 &= \sum_{i=1}^N \| z(t_i) - A \mathbb J_{2k}^T A^T X \mathbb J_{2n} X z(t_i) \|_X^2 \\
	&= \sum_{i=1}^N \| X^{1/2}z(t_i) - X^{1/2} A \mathbb J_{2k}^T A^T X \mathbb J_{2n} X z(t_i) \|_2^2 \\
	&= \| X^{1/2} S - X^{1/2} A \mathbb J_{2k}^T A^T X \mathbb J_{2n} X S \|_F^2 \\
	&= \| \tilde S - \tilde A \tilde A ^+ \tilde S \|_F^2.
\end{aligned}
\end{equation}
Here $\tilde S = X^{1/2} S$, $\tilde A = X^{1/2} A$, and $\tilde A^+ = \mathbb J_{2k}^T \tilde A^T J_{2n}$ is the symplectic inverse of $\tilde A$ with respect to the skew-symmetric matrix $J_{2n} = X^{1/2} \mathbb J_{2n} X^{1/2}$. Note that the symplectic inverse in (\ref{eq:normmor.4}) is a generalization of the symplectic inverse introduced in \cref{sec:mor.2}. Therefore, we may use the same notation (the superscript $+$) for both. We summarized the properties of this generalization in \Cref{thm:2}. With this notation, the condition (\ref{eq:normmor.2}) turns into $\tilde A ^+ \tilde A = I_{2k}$ which is equivalent to $\tilde A ^T J_{2n} \tilde A = \mathbb J_{2k}$. In other words, this condition implies that $\tilde A$ has to be a $J_{2n}$-symplectic matrix. Finally we can rewrite the minimization (\ref{eq:normmor.3}) as
\begin{equation} \label{eq:normmor.5}
\begin{aligned}
& \underset{\tilde A\in \mathbb{R}^{2n\times 2k}}{\text{minimize}}
& & \| \tilde S - P^\text{symp}_{X,\tilde A}(\tilde S) \|_F, \\
& \text{subject to}
& & \tilde A^T J_{2n} \tilde A = \mathbb J_{2k}.
\end{aligned}
\end{equation}
where $P^\text{symp}_{X,\tilde A} = \tilde A \tilde A^+$ is the symplectic projection with respect to the $X$-norm onto the colspan($\tilde A$). At first glance, the minimization (\ref{eq:normmor.5}) might look similar to (\ref{eq:mor.13}). However, since $\tilde A$ is $J_{2n}$-symplectic, and the projection operator depends on $X$, we need to seek an alternative approach to find a near optimal solution to (\ref{eq:normmor.5}). 

As (\ref{eq:mor.13}), direct approaches to solving (\ref{eq:normmor.5}) are impractical. Furthermore, there are no SVD-type methods known to the authors, that solve (\ref{eq:normmor.5}). However, the greedy generation of the symplectic basis can be generalized to generate a near optimal basis $\tilde A$. The generalized greedy method is discussed in \cref{sec:normmor.2}.

Now suppose that a basis $A=X^{-1/2}\tilde A$, with $\tilde A$ solving (\ref{eq:normmor.5}), is available such that $z \approx Ay$ with $y\in \mathbb R^{2k}$, the expansion coefficients of $z$ in the basis of $A$. Using (\ref{eq:normmor.2}) we may write the reduced system to (\ref{eq:mor.8}) as
\begin{equation} \label{eq:normmor.6}
	\dot y = \mathbb J_{2k}^T A^T X \mathbb J_{2n} X \mathbb{J}_{2n} LAy + \mathbb J_{2k}^T A^T X \mathbb J_{2n} X \mathbb{J}_{2n} \nabla_z f(Ay).
\end{equation}
Since $(\mathbb J_{2k}^T A^T X \mathbb J_{2n} X) A = I_{2k}$, we may use the chain rule to write
\begin{equation} \label{eq:normmor.7}
	\nabla_z H(z) = ( \mathbb J_{2k}^T A^T X \mathbb J_{2n} X )^T \nabla_y H(Ay).
\end{equation}
Finally, as $\nabla_z H(z) = Lz + \nabla_z f(z)$, the reduced system (\ref{eq:normmor.6}) becomes
\begin{equation} \label{eq:normmor.8}
\left\{
\begin{aligned}
	\dot y(t) &= J_{2k} A^T L A y + J_{2k} \nabla_y f(Ay), \\
	y(0) &= \mathbb J_{2k}^T A^T X \mathbb J_{2n} X z_0,
\end{aligned}
\right.
\end{equation}
where $J_{2k}=\tilde A^+ J_{2n} (\tilde A^+)^T$ is a skew-symmetric matrix. The system (\ref{eq:normmor.8}) is a generalized Hamiltonian system with the Hamiltonian defined as $\mathcal H(y) = \frac 1 2 y^TA^TLAy + f(Ay)$. Therefore, a Poisson integrator preserves the symplectic symmetry associated with (\ref{eq:normmor.8}).

We close this section by summarizing the properties of the symplectic inverse in the following theorem.
\begin{theorem} \label{thm:2}
Let $A\in \mathbb R^{2n\times 2k}$ be a $J_{2n}$-symplectic basis where $J_{2n}\in\mathbb R^{2n\times 2n}$ is a full rank and skew-symmetric matrix. Furthermore, suppose that $A^{+} = \mathbb{J}_{2k}^T A^T J_{2n}$ is the symplectic inverse. Then the following holds:
\begin{enumerate}
\item $A^+A = I_{2k}$.
\item $(A^+)^T$ is $J_{2n}^{-1}$-symplectic.
\item $\left(\left(\left(A^+\right)^T\right)^+\right)^T = A$.
\item Let $J_{2n}=X^{1/2}\mathbb J_{2n} X^{1/2}$. Then $A$ is ortho-normal with respect to the $X$-norm, if and only if $(A^+)^T$ is ortho-normal with respect to the $X^{-1}$-norm.
\end{enumerate}
\end{theorem}
\begin{proof}
It is straightforward to show all statements using the definition of a symplectic basis.
\end{proof}

\subsection{Stability Conservation} It is shown in \cite{doi:10.1137/140978922,doi:10.1137/17M1111991} that a Hamiltonian reduced system constructed by the projection $P^{\text{symp}}_{I,A}$ preserves the stability of stable equilibrium points of \cref{eq:mor.12}, and therefore, preserves the overall dynamics. In this section, we discuss that the stability of equilibrium points is also conserved using the projection operator $P^{\text{symp}}_{X,\tilde A}$.

\begin{proposition} \label{prop:new1}
\cite{bhatia2002stability} An equilibrium point $z_e\in \mathbb R^{2n}$ is Lyapunov stable if there exists a scalar function $W:\mathbb R^{2n} \to \mathbb R$ such that $\nabla W(z_e) = 0$, $\nabla^2 W(z_e)$ is positive definite, and that for any trajectory $z(t)$ defined in the neighborhood of $z_e$, we have $\frac{d}{dt} W(z(t))\leq 0$. Here $\nabla^2 W$ is the Hessian matrix of $W$, and $W$ is commonly referred to as a Lyapunov function.
\end{proposition}

It is shown in \cite{doi:10.1137/17M1111991} that the stable points of the Hamiltonian reduced system constructed using the projection $P^{\text{symp}}_{X,\tilde A}$ is Lyapunov stable. However, since the proof only requires the conservation of the Hamiltonian and the positive definiteness of $\mathcal H$, the proof also holds for generalized Hamiltonian reduced systems.

\begin{theorem}
\cite{doi:10.1137/17M1111991} Consider a Hamiltonian system of the form \cref{eq:mor.8} together with the reduced system \cref{eq:normmor.8}. Suppose that $z_e$ is an equilibrium point for \cref{eq:mor.8} and that $y_e = \tilde A ^+ X^{1/2} z_e$. If $H$ (or $-H$) is a Lyapunov function satisfying \cref{prop:new1}, then $z_e$ and $y_e$ are Lyapunov stable equilibrium points for \cref{eq:mor.8} and \cref{eq:normmor.8}, respectively.
\end{theorem}

\subsection{Greedy generation of a $J_{2n}$-symplectic basis} \label{sec:normmor.2}
In this section we modify the greedy algorithm introduced in \cref{sec:mor.3} to construct a $J_{2n}$-symplectic basis. Ortho-normalization is an essential step in greedy approaches to basis generation \cite{hesthaven2015certified,quarteroni2015reduced}. Here, we summarize a variation of the GS orthogonalization process, known as the \emph{symplectic GS} process.

Suppose that $\Omega_{J_{2n}}$ is a symplectic form defined on $\mathbb R^{2n}$ such that $\Omega_{J_{2n}}(x,y) = x^T J_{2n} y$, for all $x,y\in \mathbb R^{2n}$ and some full rank and skew-symmetric matrix $J_{2n} = X^{1/2} \mathbb J_{2n} X^{1/2}$. We would like to build a basis of size $2k+2$ in an iterative manner and start with some initial vector, e.g. $e_1 = z_0/\| z_0 \|_X$. It is known that a symplectic basis has an even number of basis vectors \cite{Marsden:2010:IMS:1965128}. We may take $Te_1$, where $T = X^{-1/2} \mathbb J_{2n}^{T}X^{1/2}$, as a candidate for the second basis vector. It is easily verified that $\tilde A_2=[e_1|Te_1]$ is $J_{2n}$-symplectic and consequently, $\tilde A_2$ is the first basis generated by the greedy approach. Next, suppose that $\tilde A_{2k} = [e_1|\dots|e_k|Te_1|\dots|Te_k]$ is generated in the $k$th step of the greedy method and $z\not \in \text{colspan}\left(\tilde A_{2k}\right)$ is provided. We aim to $J_{2n}$-orthogonalize $z$ with respect to the basis $\tilde A_{2k}$. This means we seek a coefficient vector $\alpha\in \mathbb R^{2k}$ such that
\begin{equation} \label{eq:normmor.9}
	\Omega_{J_{2n}}\left( z + \tilde A_{2k}\alpha ,  y \right) = 0,
\end{equation}
for all possible $y \in \text{colspan}(\tilde A_{2k})$. It is easily checked that \eqref{eq:normmor.9} has the unique solution $\alpha_i = -\Omega_{J_{2n}}(z,Te_i)$ for $i\leq k$ and $\alpha_i = \Omega_{J_{2n}}(z,e_i)$ for $i>k$, i.e., $z$ has a unique symplectic projection. If we take $\tilde z = z + \tilde A_{2k}\alpha$, then the next candidate pair of basis vectors are $e_{k+1} = \tilde z / \| \tilde z \|_X$ and $Te_{k+1}$. Finally, the basis generated at the $(k+1)$-th step of the greedy method is given by
\begin{equation} \label{eq:normmor.11}
	\tilde A_{2k+2} = [e_1|\dots|e_{k+1}|Te_1|\dots|Te_{k+1}].
\end{equation}
\Cref{thm:3} guarantees that the column vectors of $\tilde A_{2k+2}$ are linearly independent. Furthermore, it is checked easily that $\tilde A_{2k+2}$ is $J_{2n}$-symplectic. We note that the symplectic GS orthogonalization process is chosen due to its simplicity. However, in problems where there is a need for a large basis, this process might be impractical. In such cases, one may use a backward stable routine, e.g. the isotropic Arnoldi method or the isotropic Lanczos method \cite{doi:10.1137/S1064827500366434}.

It is well known that a symplectic basis, in general, is not norm bounded \cite{doi:10.1137/050628519}. The following theorem guarantees that the greedy method for generating a $J_{2n}$-symplectic basis yields a bounded basis.
\begin{theorem} \label{thm:3}
The basis generated by the greedy method for constructing a $J_{2n}$-symplectic basis is orthonormal with respect to the $X$-norm.
\end{theorem}
\begin{proof}
Let $\tilde A_{2k}=[e_1|\dots,e_k|Te_1|\dots|Te_k]$ be the $J_{2n}$-symplectic basis generated at the $k$th step of the greedy method. Using the fact that $\tilde A_{2k}$ is $J_{2n}$-symplectic, one can check that
\begin{equation} \label{eq:normmor.12}
	\left\langle e_i,e_j\right\rangle_X = \left\langle Te_i,Te_j\right\rangle_X = \Omega_{J_{2n}}(e_i,Te_j)=\delta_{i,j}, \quad i,j=1,\dots,k,	
\end{equation}
and
\begin{equation} \label{eq:normmor.13}
	\left\langle e_i,Te_j \right\rangle_X = \Omega_{J_{2n}}(e_i,e_j) = 0\quad i,j=1,\dots,k,
\end{equation}
where $\delta_{i,j}$ is the Kronecker delta function. This ensures that $\tilde A_{2k}^TX\tilde A_{2k} = I_{2k}$, i.e., $\tilde A_{2k}$ is an ortho-normal basis with respect to the $X$-norm.
\end{proof}
We note that if we take $X=I_{2n}$, then the greedy process generates a $\mathbb J_{2n}$- symplectic basis. With this choice, the greedy method discussed above becomes identical to the greedy process discussed in \cref{sec:mor.3}. Therefore, the symplectic model reduction with a weight matrix $X$ is indeed a generalization of the method discussed in \cref{sec:mor.2}.

We notice that $X^{1/2}$ does not explicitly appear in \cref{eq:normmor.8}. Therefore, it is desirable to compute $A_{2k} = X^{-1/2} \tilde A_{2k}$ without requiring the computation of the matrix square root of $X$. It is easily checked that the matrix $B_{2k}:=X^{1/2} \tilde A_{2k} = XA_{2k}$ is $\mathbb J_{2n}$-symplectic and orthonormal. Reformulation of condition \cref{eq:normmor.9} yields
\begin{equation} \label{eq:normmor.13.1}
	\Omega_{\mathbb J_{2n}}\left( w + B_{2k} \alpha, \bar y \right) = 0, \quad \forall \bar y \in \text{colspan}(B_{2k}),
\end{equation}
where $w = X^{1/2}z$. From \cref{eq:mor.14.1} we know that \cref{eq:normmor.13.1} has the unique solution $\alpha_i = - \Omega_{\mathbb J_{2n}}(z,\mathbb J_{2n}^T \hat e_i)$ for $i\leq k$ and $\alpha_i = \Omega_{\mathbb J_{2n}}(z,\hat e_i)$ for $i>k$, where $\hat e_i$ is the $i$th column vector of $B_{2k}$. Furthermore, we take 
\begin{equation}
	\hat e_{k+1} = \hat z / \| \hat z \|_2, \quad \hat z = w + B_{2k} \alpha,
\end{equation}
as the next enrichment vector to construct
\begin{equation}
	B_{2(k+1)} = [ \hat e_1 | \dots | \hat e_{k+1} | \mathbb J_{2n}^T \hat e_1 | \dots | \mathbb J_{2n}^T \hat e_{k+1} ].
\end{equation}
One can recover $e_{k+1}$ form the relation $e_{k+1} = X^{-1/2} \hat e_{k+1}$. However, since we are interested in the matrix $A_{2(k+1)}$ and not $\tilde A_{2(k+1)}$, we can solve the system $XA_{2(k+1)} = B_{2(k+1)}$ for $A_{2(k+1)}$. This procedure eliminates the computation of $X^{1/2}$.

For identifying the best vectors to be added to a set of basis vectors, we may use similar error functions to those introduced in \cref{sec:mor.3}. The projection error can be used to identify the snapshot that is worst approximated by a given basis $\tilde A_{2k}$:
\begin{equation} \label{eq:normmor.14}
\begin{aligned}
	z_{k+1} &:= \underset{z\in\{ z(t_i)\}_{i=1}^{N}}{\text{argmax } }\| z - P(z) \|_X.
\end{aligned}
\end{equation}
Where $P$ is defined in (\ref{eq:normmor.1}). Alternatively we can use the loss in the Hamiltonian function in (\ref{eq:mor.16}) for parameter dependent problems. We summarize the greedy method for generating a $J_{2n}$-symplectic matrix in \Cref{alg:2}.

\begin{algorithm} 
\caption{The greedy algorithm for generation of a $J_{2n}$-symplectic basis} \label{alg:2}
{\bf Input:} Tolerated projection error $\delta$, initial condition $ z_0$, the snapshots $\mathcal Z = \{Xz(t_i)\}_{i=1}^{N}$, full rank matrix $X=X^T>0$
\begin{enumerate}
\item $z_1 = Xz(0)$
\item $P = A \mathbb J_{2k}^T A^T X \mathbb J_{2n}$
\item $\hat e_1 \leftarrow z_1/ \| z_1 \|_2$
\item $B \leftarrow [\hat e_1| \mathbb J_{2n}^T \hat e_1]$
\item $k \leftarrow 1$
\item \textbf{while} $\| z - Pz\|_X > \delta$ for any $z \in \mathcal Z$
\item \hspace{0.5cm} $z_{k+1} := \underset{z\in \mathcal Z}{\text{argmax }} \| z - Pz \|_X$
\item \hspace{0.5cm} $\mathbb J_{2n}$-orthogonalize $z_{k+1}$ to obtain $\hat e_{k+1}$
\item \hspace{0.5cm} $B \leftarrow [\hat e_1|\dots |\hat e_{k+1} | \mathbb J_{2n}^T \hat e_1|\dots| \mathbb J_{2n}^T  \hat e_{k+1}]$
\item \hspace{0.5cm} $k \leftarrow k+1$
\item \textbf{end while}
\item solve $X A = B$ for $A$
\end{enumerate}
\vspace{0.5cm}
{\bf Output:} The reduced basis $A$
\end{algorithm}
It is shown in \cite{doi:10.1137/17M1111991} that under natural assumptions on the solution manifold of (\ref{eq:mor.8}), the original greedy method for symplectic basis generation converges exponentially fast. We expect the generalized greedy method, equipped with the error function (\ref{eq:normmor.14}), to converge as fast, since the $X$-norm is topologically equivalent to the standard Euclidean norm \cite{friedman1970foundations}, for a full rank matrix $X$.

\subsection{Efficient evaluation of nonlinear terms} \label{sec:normmor.3}
The evaluation of the nonlinear term in (\ref{eq:normmor.8}) still retains a computational complexity proportional to the size of the full order system (\ref{eq:mor.8}). To overcome this, we take an approach similar to \cref{sec:mor.2}. The DEIM approximation of the nonlinear term in (\ref{eq:normmor.8}) yields
\begin{equation} \label{eq:normmor.15}
	\dot y = J_{2k} A^TLAy + \tilde A ^+ X^{1/2} \mathbb J_{2n} U (\mathcal P^TU)^{-1}\mathcal  P^T \nabla_z f(Ay).
\end{equation}
Here $U$ is a basis constructed from the nonlinear snapshots $\{\nabla_z f(z(t_i))\}_{i=1}^N$, and $\mathcal P$ is the interpolating index matrix \cite{Chaturantabut:2010cz}. As discussed in \cref{sec:mor.2}, for a general choice of $U$, the reduced system (\ref{eq:normmor.8}) does not retain a Hamiltonian form. Since $(\tilde A^+ X^{1/2}) A = I_{2k}$ applying the chain rule on (\ref{eq:normmor.15}) yields
\begin{equation} \label{eq:normmor.16}
	\dot y = J_{2k} A^TLAy + \tilde A ^+ X^{1/2} \mathbb J_{2n} U (\mathcal P^TU)^{-1} \mathcal P^T (\tilde A^+ X^{1/2})^T \nabla_y f(Ay).
\end{equation}
Freedom in the choice of the basis $U$ allows us to require $U = X^{1/2} (\tilde A^+)^T$. This reduces the complex expression in (\ref{eq:normmor.16}) to
\begin{equation} \label{eq:normmor.17}
	\dot y = J_{2k} A^TLAy + J_{2k} \nabla_y f(Ay),
\end{equation}
and hence we recover the Hamiltonian structure. The reduced system then yields
\begin{equation} \label{eq:normmor.18}
\left\{
\begin{aligned}
	\dot y(t) &= J_{2k} A^TLAy + J_{2k} (\mathcal P^TX \mathbb J_{2n} X A \mathbb J_{2k})^{-1} \mathcal P^T \nabla_z f(z), \\
	y(0) &= \mathbb J_{2k}^T A^T X J X z_0.
\end{aligned}
\right.
\end{equation}
We now discuss how to ensure that $X^{1/2} (\tilde A^+)^T$ is a basis for the nonlinear snapshots. Note that if $z \in \text{colspan}\left(X^{1/2} (\tilde A^+)^T\right)$ then $X^{-1/2} z \in \text{colspan}\left(( \tilde A^+)^T \right)$. Therefore, it is sufficient to require $(\tilde A^+)^T$ to be a basis for $\{X^{-1/2} \nabla_z f(z(t_i))\}_{i=1}^N$. \Cref{thm:2} suggests that $(\tilde A^+)^T$ is a $J_{2n}^{-1}$-symplectic basis and that the transformation between $\tilde A$ and $(\tilde A^+)^T $ does not affect the symplectic feature of the bases. Consequently, from $A$ we may compute $(\tilde A^+)^T$ and enrich it with snapshots $\{X^{-1/2} \nabla_z f(z(t_i))\}_{i=1}^N$. Once $(\tilde A^+)^T$ represents the nonlinear term with the desired accuracy, we may compute $\tilde A= \left( \left( ( \tilde A^+ )^T \right)^+ \right)^T$ to obtain the reduced basis for (\ref{eq:normmor.18}). \Cref{thm:2} implies that $(\tilde A^+)^T$ is ortho-normal with respect to the $X^{-1}$-norm. This affects the ortho-normalization process. We note that greedy approaches to basis generation do not generally result in a minimal basis.

As discussed in \cref{sec:normmor.2} it is desirable to eliminate the computation of $X^{\pm 1/2}$. Having $z \in \text{colspan}\left(X^{1/2} (\tilde A^+)^T\right)$ implies that $X^{-1} z \in \text{colspan}(\mathbb J_{2n}^T X A \mathbb J_{2n})$. Note that \Cref{alg:2} constructs a $\mathbb J_{2n}$-symplectic matrix $XA$ and $\mathbb J_{2n}^T X A \mathbb J_{2n}$ is the symplectic inverse of $XA$ with respect to the standard symplectic matrix $\mathbb J_{2n}$. Given $e$ as a candidate for enriching $X^{1/2} (\tilde A^+)^T$ we may instead enrich $\mathbb J_{2n}^T X A \mathbb J_{2n}$ with $\hat e$, that solves $X \hat e = e$.

Since  $\mathbb J_{2n}^T X A \mathbb J_{2n}$ is $\mathbb J_{2n}$-symplectic the projection operator onto the column span of $\mathbb J_{2n}^T X A \mathbb J_{2n}$ can be constructed as $Q=\mathbb J_{2n}^T X A \mathbb J_{2n}A^TX$. Given a nonlinear snapshot $z$, we may need to project the vector $X^{-1}z$ onto colspan$(\mathbb J_{2n}^T X A \mathbb J_{2n})$. However, $Q(X^{-1}z)=\mathbb J_{2n}^T X A \mathbb J_{2n}A^Tz$ and thus, the matrix $X^{-1}$ does not appear explicitly. This process eliminates the computation of $X^{\pm 1/2}$. We summarize the process of generating a basis for the nonlinear terms in \Cref{alg:3}.

\begin{algorithm} 
\caption{Generation of a basis for nonlinear terms} \label{alg:3}
{\bf Input:} Tolerated projection error $\delta$, $\mathbb J_{2n}$-symplectic basis $B = X A$ of size $2k$, the snapshots $\mathcal G = \{ \nabla_zf(z(t_i))\}_{i=1}^{N}$, full rank matrix $X=X^T>0$
\begin{enumerate}
\item $Q \leftarrow \mathbb J_{2n}^T X A \mathbb J_{2n}A^T$
\item compute $(B^+)^T = \mathbb J_{2n}^T B \mathbb J_{2n} = [e_1|\dots |e_{k} | \mathbb J_{2n}^Te_1|\dots| \mathbb J_{2n}^Te_{k}]$
\item \textbf{while} $\| g - Q g \|_2 > \delta$ for any $g \in \mathcal G$
\item \hspace{0.5cm} $g_{k+1} := \underset{g\in \mathcal G}{\text{argmax }} \| g -  Q g  \|_2$
\item \hspace{0.5cm} solve $X e = g_{k+1}$ for $e$
\item \hspace{0.5cm} $\mathbb J_{2n}$-orthogonalize $e$ to obtain $e_{k+1}$
\item \hspace{0.5cm} $(B^+)^T \leftarrow [e_1|\dots |e_{k+1} | \mathbb J_{2n}^Te_1|\dots| \mathbb J_{2n}^Te_{k+1}]$
\item \hspace{0.5cm} $k \leftarrow k+1$
\item \textbf{end while}
\item compute $XA = \left( \left (B^+)^T \right)^+ \right)^T$
\end{enumerate}
\vspace{0.5cm}
{\bf Output:} $\mathbb J_{2n}$-symplectic basis $XA$
\end{algorithm}

\subsection{Offline/online decomposition} \label{sec:normmor.4}
Model order reduction becomes particularly useful for parameter dependent problems in multi-query settings. For the purpose the of most efficient computation, it is important to delineate high dimensional ($\mathcal{O}(n^{\alpha})$) offline computations from low dimensional ($\mathcal{O}(k^{\alpha})$) online ones, for some $\alpha \in \mathbb N$. Time intensive high dimensional quantities are computed only once for a given problem in the offline phase and the cheaper low dimensional computations can be performed in the online phase. This segregation or compartmentalization of quantities, according to their computational cost, is referred to as the offline/online decomposition.

More precisely, one can decompose the computations into the following stages:
\emph{Offline stage:} Quantities in this stage are computed only once and then used in the online stage.
\begin{enumerate}
\item Generate the weighted snapshots $\{ X z(t_i) \}_{i=1}^N$ and the snapshots of the nonlinear term $\{\nabla_zf(z(t_i))\}_{i=1}^N$
\item Generate a $J_{2n}$-symplectic basis for the solution snapshots and the snapshots of the nonlinear terms, following \Cref{alg:2,alg:3}, respectively.
\item Assemble the reduced order model \cref{eq:normmor.18}.
\end{enumerate}
\emph{Online stage:} The reduced model \cref{eq:normmor.18} is solved for multiple parameter sets and the output is extracted.

\section{Numerical results} \label{sec:res}
Let us now discuss the performance of the symplectic model reduction with a weighted inner product. In \cref{sec:res.1,sec:res.1.1} we apply the model reduction to equations of a vibrating elastic beam without and with cavity, respectively. And we examine the evaluation of the nonlinear terms in the model reduction of the sine-Gordon equation, in section \cref{sec:res.2}.

\subsection{The elastic beam equation} \label{sec:res.1}
Consider the equations governing small deformations of a clamped elastic body $\Gamma\subset \mathbb R^{3}$ as 
\begin{equation} \label{eq:res.1}
\left\{
\begin{aligned}
	u_{tt}(t,x) &= \nabla \cdot \sigma + f, \quad & x\in \Gamma, \\
	u(0,x) &= \vec 0, & x\in \Gamma,\\
	\sigma \cdot n &= \tau, & x \in \partial \Gamma_\tau,\\
	u(t,x) &= \vec 0, & x \in\partial \Gamma \backslash \partial \Gamma_\tau,
\end{aligned}
\right.
\end{equation}
and
\begin{equation}  \label{eq:res.2}
	\sigma = \lambda (\nabla \cdot u) I + \mu(\nabla u + (\nabla u)^T).
\end{equation}
Here $u:\Gamma \to \mathbb{R}^3$ is the unknown displacement vector field, subscript $t$ denotes derivative with respect to time, $\sigma:\Gamma \to \mathbb{R}^{3\times 3}$ is the stress tensor, $f$ is the body force per unit volume, $\lambda$ and $\mu$ are Lam\'e's elasticity parameters for the material in $\Gamma$, $I$ is the identity tensor, $n$ is the outward unit normal vector at the boundary and $\tau:\partial \Gamma_\tau \to \mathbb R^3$ is the traction at a subset of the boundary $\partial \Gamma_\tau$ \cite{langtangen2017solving}. We refer to \Cref{fig:0}(a) for a snapshot of the elastic beam.

\begin{figure}[t] \label{fig:0}
\begin{tabular}{cc}
\includegraphics[width=0.45\textwidth]{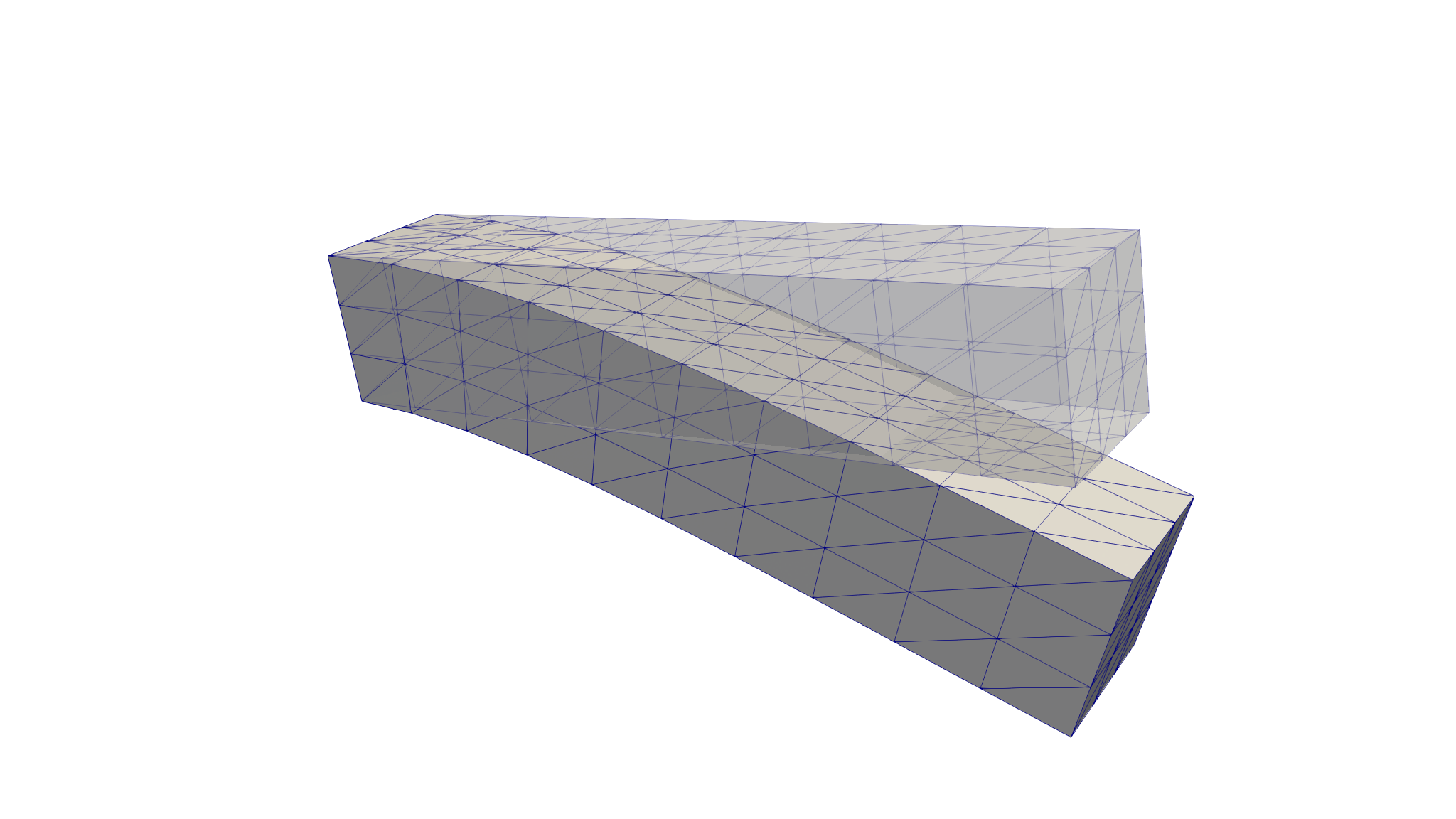} & \includegraphics[width=0.45\textwidth]{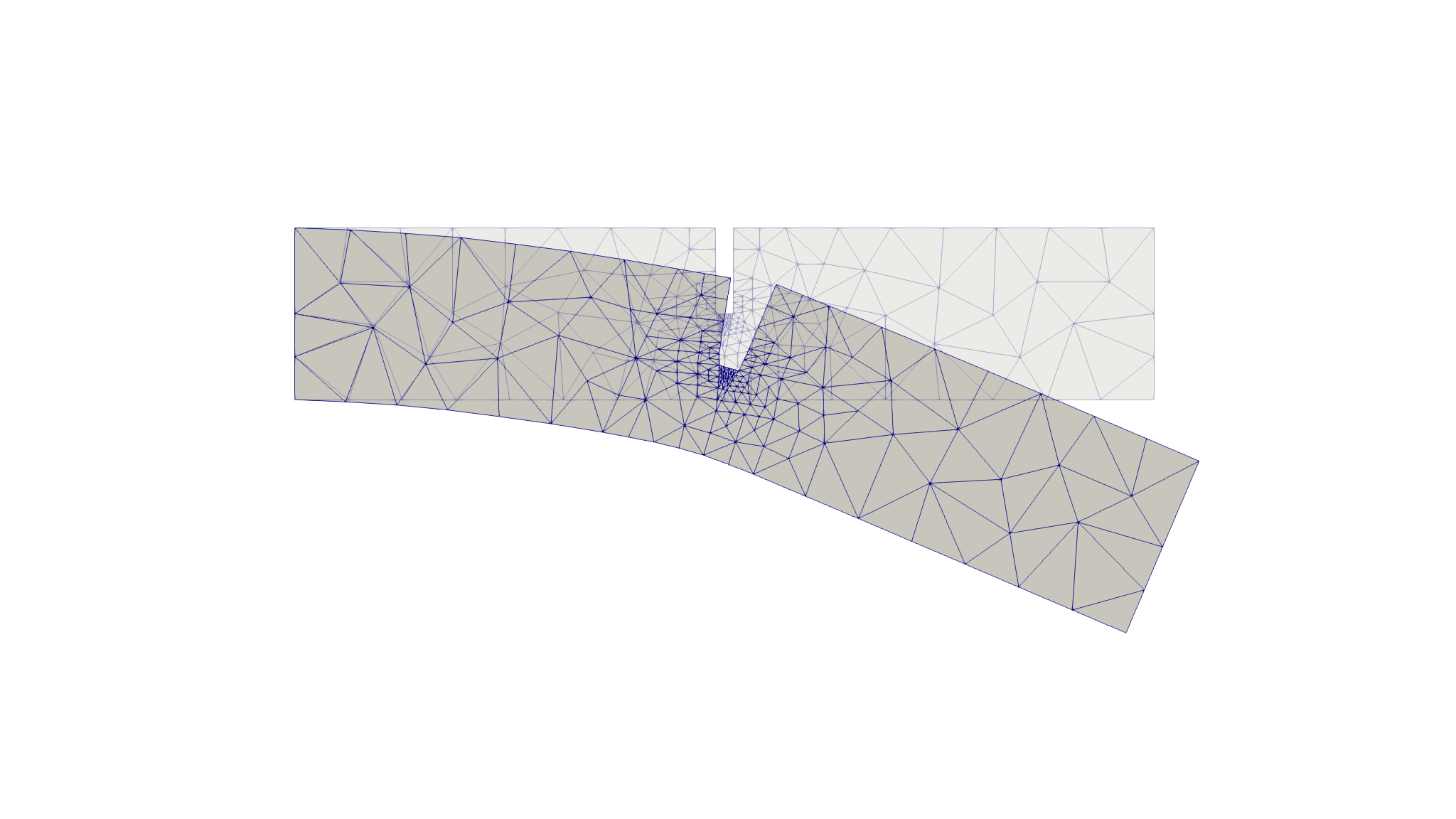} \\
(a) & (b)
\end{tabular}
\caption{(a) initial condition and a snapshot of the 3D beam. (b) initial condition and a snapshot of the 2D beam with cavity. }
\end{figure}

We define a vector valued function space as $V = \{ u \in (L^2(\Gamma))^3 : \| \nabla u_i \|_2 \in L^2 \text{, } i=1,2,3\text{, } u = \vec 0 \text{ on } \partial \Gamma_\tau \}$, equipped with the standard $L^2$ inner product $(\cdot,\cdot):V\times V \to \mathbb R$, and seek the solution to (\ref{eq:res.1}). To derive the weak formulation of (\ref{eq:res.1}), we multiply it with the vector valued test function $v \in V$, integrate over $\Gamma$, and use integration by parts to get 
%\begin{equation}  \label{eq:res.3}
%	\int_{\Gamma} u_{tt} \cdot v\ dx = \int_{\Gamma} (\nabla \cdot \sigma) \cdot v \ dx + \int_{\Gamma} f \cdot v \ dx.
%\end{equation}
%We may integrate the first term on the right hand side by parts to obtain
\begin{equation}  \label{eq:res.4}
	\int_{\Gamma} u_{tt} \cdot v\ dx = - \int_{\Gamma} \sigma : \nabla v \ dx+ \int_{\partial \Gamma_\tau} (\sigma \cdot n) \cdot v\ ds +  \int_{\Gamma} f \cdot v \ dx,
\end{equation}
where $\sigma : \nabla v = \sum_{i,j}\sigma_{ij}(\nabla v)_{ji}$ is the tensor inner product. Note that the skew-symmetric part of $\nabla v$ vanishes over the product $\sigma : \nabla v$, since $\sigma$ is symmetric. By prescribing the boundary conditions to (\ref{eq:res.4}) we recover
\begin{equation} \label{eq:res.5}
	\int_{\Gamma} u_{tt} \cdot v\ dx = - \int_{\Gamma} \sigma : \text{Sym}(\nabla v) \ dx+ \int_{\partial \Gamma_\tau} \tau \cdot v\ ds +  \int_{\Gamma} f \cdot v \ dx,
\end{equation}
with Sym$(\nabla v) = (\nabla v + (\nabla v)^T)/2$. The variational form associated to (\ref{eq:res.1}) is
\begin{equation} \label{eq:res.6}
	(u_{tt},v) = - a(u,v) + b(v), \quad u,v\in V,
\end{equation}
where
\begin{equation} \label{eq:res.7}
\begin{aligned}
	a(u,v) &= \int_{\Gamma} \sigma : \text{Sym}(\nabla v) \ dx, ~
	b(v) &= \int_{\partial \Gamma_\tau} \tau \cdot v\ ds +  \int_{\Gamma} f \cdot v \ dx.
\end{aligned}
\end{equation}
To obtain the FEM discretization of (\ref{eq:res.6}), we triangulate the domain $\Gamma$ and define vector valued piece-wise linear basis functions $\{\phi_i\}_{i=1}^{N_h}$, referred to as the \emph{hat functions}. We define the FEM space $V_h$, an approximation of $V$, as the span of those basis functions. Projecting (\ref{eq:res.6}) onto $V_h$ yields the discretized weak form
\begin{equation} \label{eq:res.8}
	((u_h)_{tt},v_h) = - a(u_h,v_h) + b(v_h),\quad u_h,v_h\in V_h.
\end{equation}
Any particular function $u_h$ can be expressed as $u_h = \sum_{i=1}^{N_h} q_i \phi_i$, where $q_i$, $i=1,\dots,N_h$, are the expansion coefficients. Therefore, by choosing test functions $v_h = \phi_i$, $i=1,\dots,N_h$, we obtain the ODE system
\begin{equation} \label{eq:res.9}
	M\ddot q = -K q + g_{q}.
\end{equation}
where $q=(q_1,\dots,q_{N_h})^T$ are unknowns, the \emph{mass matrix} $M\in \mathbb R^{N_h\times N_h}$ is given as $M_{i,j} = (\phi_i,\phi_j)$, the \emph{stiffness matrix} $K\in \mathbb R^{N_h\times N_h}$ is given as $K_{i,j} = a(\phi_j,\phi_i)$ and $g_q=(b(v_1),\dots,b(v_{N_h}))^T$. Now introduce the canonical coordinate $p = M\dot q$ to recover the Hamiltonian system
\begin{equation} \label{eq:res.10}
	\dot z = \mathbb J_{2N_h} Lz + g_{qp},
\end{equation}
where
\begin{equation} \label{eq:res.11}
	z = 
	\begin{pmatrix}
	q \\
	p	
	\end{pmatrix}, \quad 
	L = 
	\begin{pmatrix}
	K & 0 \\
	0 & M^{-1}
	\end{pmatrix}, \quad
	g_{qp} =
	\begin{pmatrix}
	0 \\
	g_q
	\end{pmatrix},
\end{equation}
together with the Hamiltonian function $H(z) = \frac{1}{2} z^TLz + z^T \mathbb J_{2N_h}^T g_{qp}$. An appropriate FEM setup leads to a symmetric and positive-definite matrix $L$. Hence, it seems natural to take $X=L$, the energy matrix associated to (\ref{eq:res.10}). The system parameters are summarized in the table below. For further information regarding the problem, we refer to \cite{langtangen2017solving}.

\vspace{0.5cm}
\begin{center}
\begin{tabular}{|l|l|}
\hline
Domain shape & box: $l_x = 1,\ l_y = 0.2,\ l_z = 0.2$ \\
Time step-size & $\Delta t = 0.01$ \\
Gravitational force & $f = (0,0,-0.4)^T$ \\
Traction & $\tau = \vec 0$ \\
Lam\'e parameters & $\lambda = 1.25$, $\mu = 1.0$ \\
Degrees of freedom & $2N_{h} = 1650$ \\
\hline
\end{tabular}
\end{center}
\vspace{0.5cm}
Projection operators $P_{X,V}$, $P^{\text{symp}}_{I,A}$ and $P^{\text{symp}}_{X,\tilde A}$ are constructed following \cref{sec:mor.1,sec:mor.2,sec:mor.3}, respectively, with $\sigma = 5\times 10^{-4}, 2\times 10^{-4}$ and $1\times 10^{-4}$. In order to apply a symplectic time integrator, we first compute the transformation $J_{2k} = \mathcal T \mathbb J_{2k} \mathcal T^T$ using the symplectic GS method with complete pivoting. The reduced systems, obtained from $P^{\text{symp}}_{I,A}$ and $P^{\text{symp}}_{X,\tilde A}$, are then integrated in time using the St\"ormer-Verlet scheme to generate the temporal snapshots. The reduced system obtained from $P_{X,V}$ is integrated using a second order implicit Runge-Kutta method. Note that the St\"ormer-Verlet scheme is not used since the canonical form of a Hamiltonian system is destroyed when $P_{X,V}$ is applied.

\begin{figure}[t] \label{fig:1}
\begin{tabular}{cc}
\includegraphics[width=0.45\textwidth]{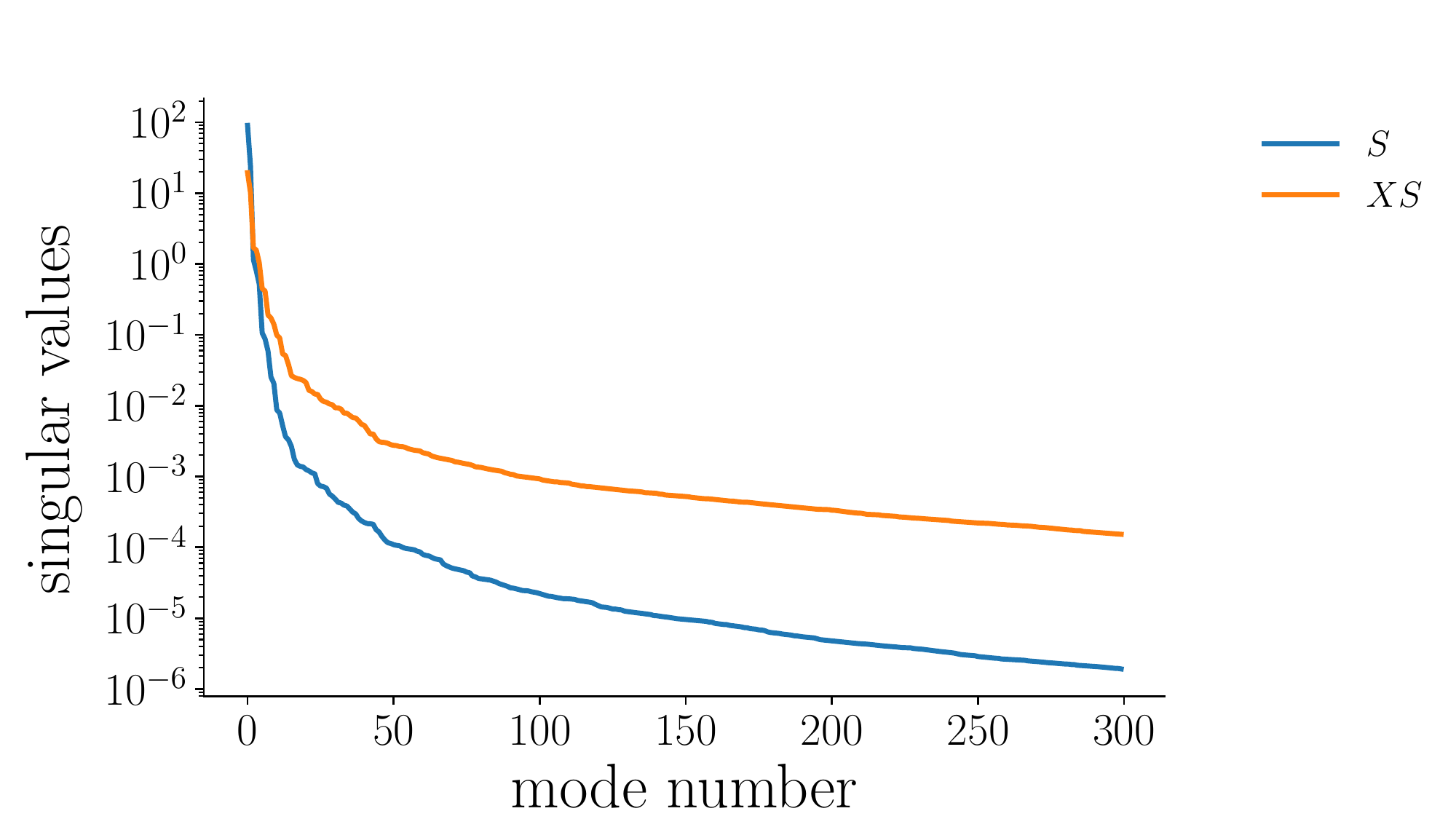} & \includegraphics[width=0.45\textwidth]{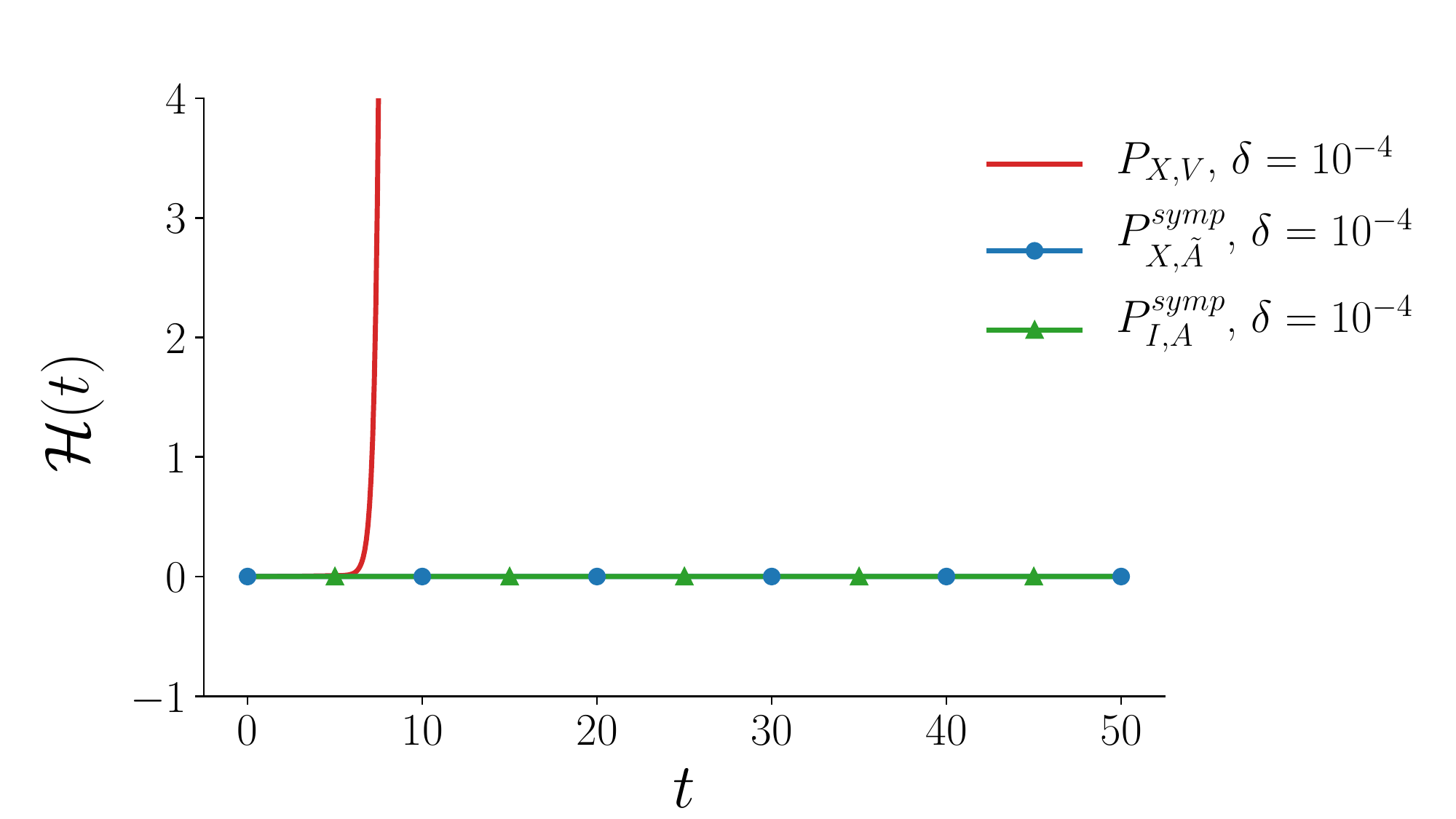} \\
(a) & (b) \\
\includegraphics[width=0.45\textwidth]{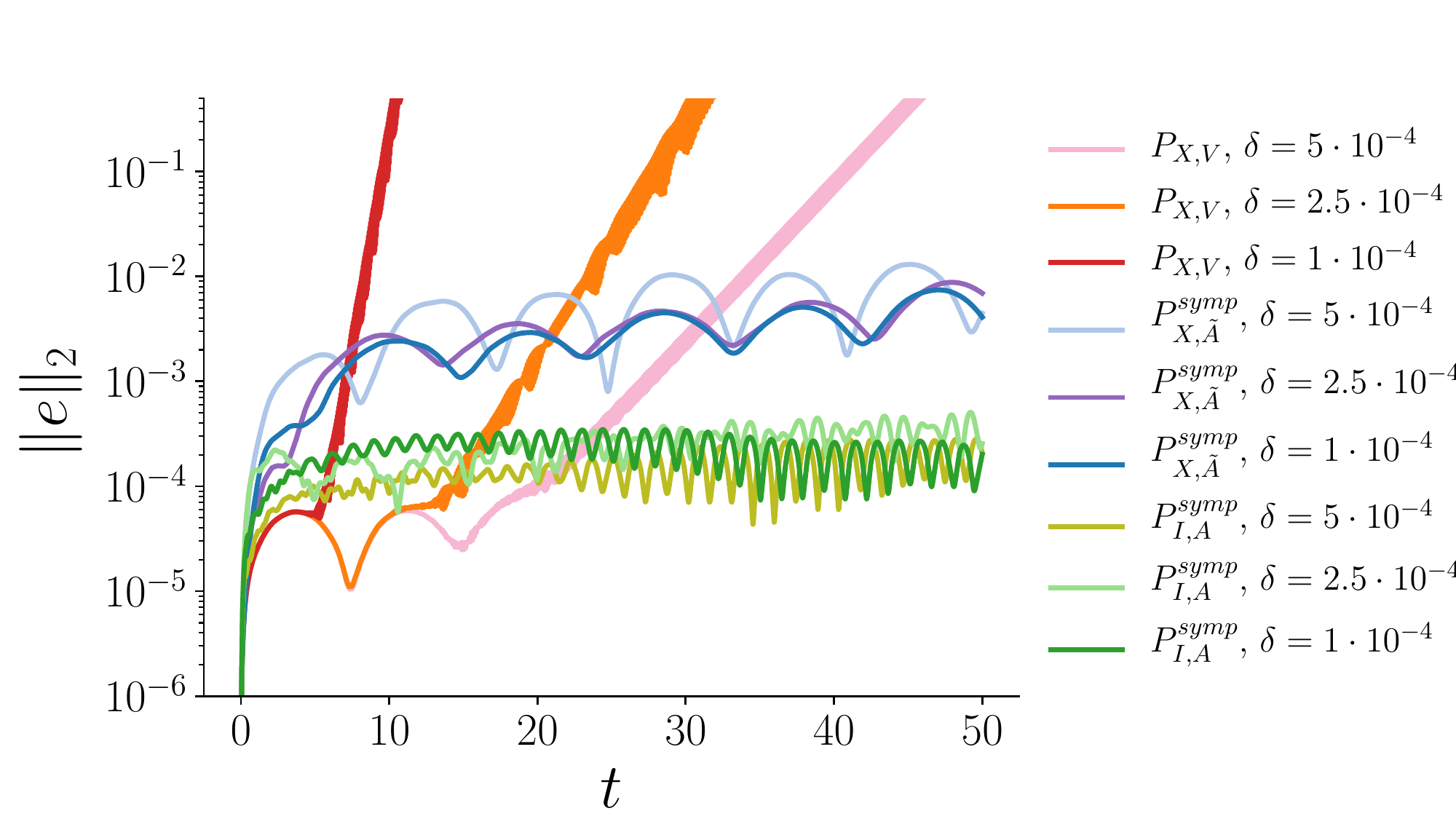} & \includegraphics[width=0.45\textwidth]{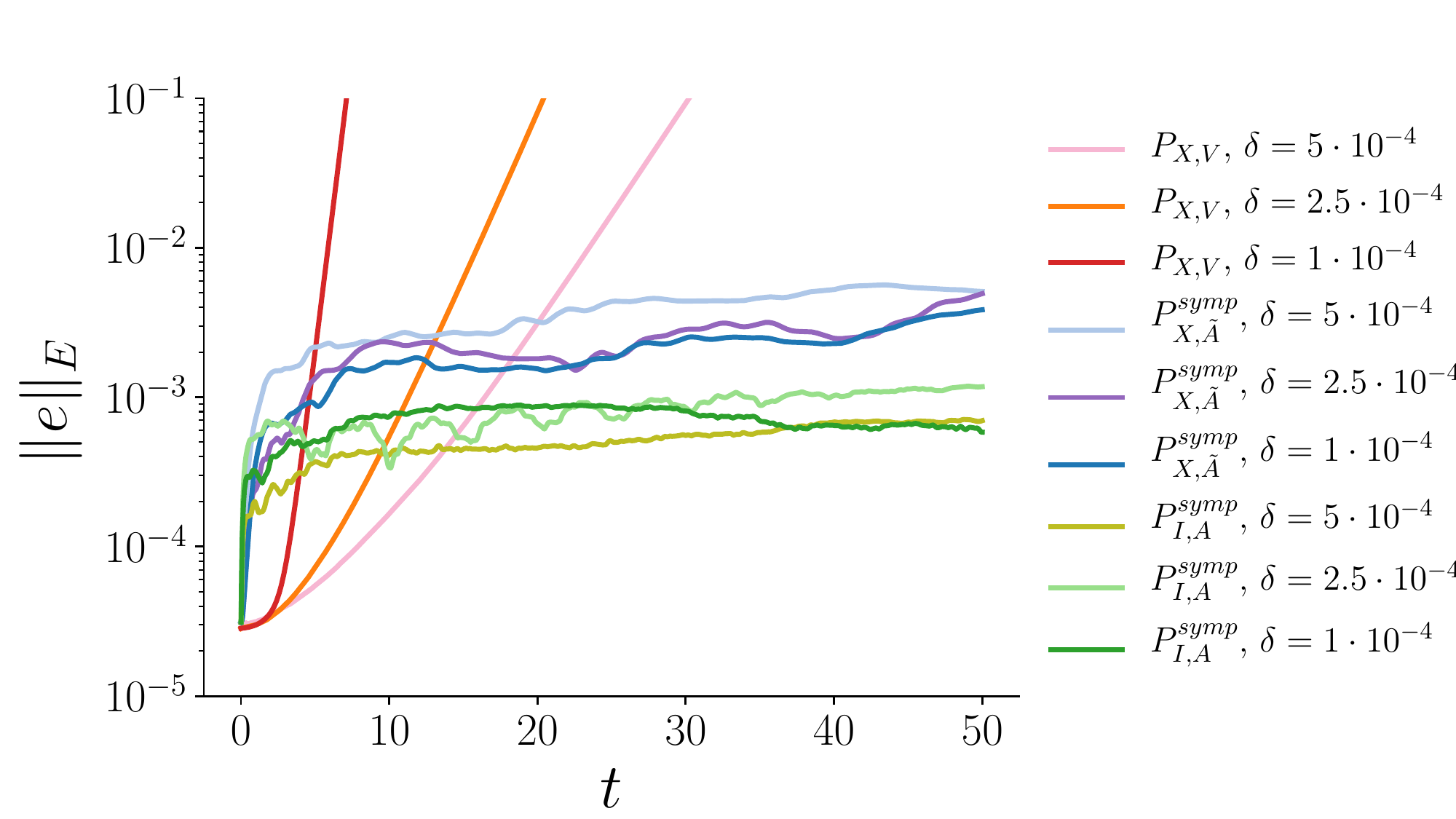} \\
(c) & (d) \\
\end{tabular}
\caption{Numerical results related to the beam equation. (a) the decay of the singular values. (b) conservation of the Hamiltonian. (c) error with respect to the 2-norm. (d) error with respect to the $X$-norm.}
\end{figure}

\Cref{fig:1}(a) shows the decay of the singular values of the temporal snapshots $S$ and $XS$, respectively. The difference in the decay indicates that the reduced systems constructed using $P_{I,A}^{\text{symp}}$ and $P_{X,\tilde A}^{\text{symp}}$ would have different sizes to achieve similar accuracy.

\Cref{fig:1}(b) shows the conservation of the Hamiltonian for the methods discussed above. This confirms that the symplectic methods preserve the Hamiltonian and the system energy. However, the Hamiltonian blows up for the reduced system constructed by the projection $P_{X,V}$.

\Cref{fig:1}(c) shows the $L^2$ error between the projected systems and the full order system, defined as
\begin{equation}
	\| e \|_{L^2} = \sqrt{(e,e)} \approx \sqrt{ (q - \hat q)^T M (q-\hat q) },
\end{equation}
where $e\in V$ is the error function and $\hat q \in \mathbb R^{2n}$ is an approximation for $q$. We notice that the reduced system obtained by the non-symplectic method is unstable and the reduced system, constructed using $P_{X,V}$, is more unstable as $k$ increases. On the other hand, the symplectic methods yield a stable reduced system. Although the system, constructed by the projection $P^{\text{symp}}_{X,\tilde A}$, is not based on the 2-norm projection, the error remains bounded with respect to the 2-norm. 

We define the energy norm $\| \cdot \|_E : V \to \mathbb R$ as
\begin{equation}
	\| (u,\dot u) \|_E = \sqrt{ a(u,u) + (\dot u , \dot u) } \approx \| z \|_X.
\end{equation}
\Cref{fig:1}(d) shows the MOR error with respect to the energy norm. We observe that the classical model reduction method based on the projection $P_{X,V}$ does not yield a stable reduced system. However, the symplectic methods provide a stable reduced system. We observe that the original symplectic approach also provides an accurate solution with respect to the energy norm. Nevertheless, the relation between the two norms depends on the problem set up and the choice of discretization \cite{DEPARIS20094359}.

\subsection{Elastic beam with cavity}  \label{sec:res.1.1}
In this section we investigate the performance of the proposed method on a two dimensional elastic beam that contains a cavity. In this case a nonuniform triangulated mesh is desirable to balance the computational cost of a FEM discretization with the numerical error around the cavity. \Cref{fig:0}(a) shows the nonuniform mesh used in this section.
System parameters are taken to be identical to those in \cref{sec:res.1}. Numerical parameters are summarized in the table below.
\vspace{0.5cm}
\begin{center}
\begin{tabular}{|l|l|}
\hline
cavity width & $l_c = 0.1$ \\
Time step-size & $\Delta t = 4\times 10^{-4}$ \\
Degrees of freedom & $2N_{h} = 744$ \\
\hline
\end{tabular}
\end{center}
\vspace{0.5cm}

\begin{figure} \label{fig:1.1}
\begin{tabular}{cc}
\includegraphics[width=0.45\textwidth]{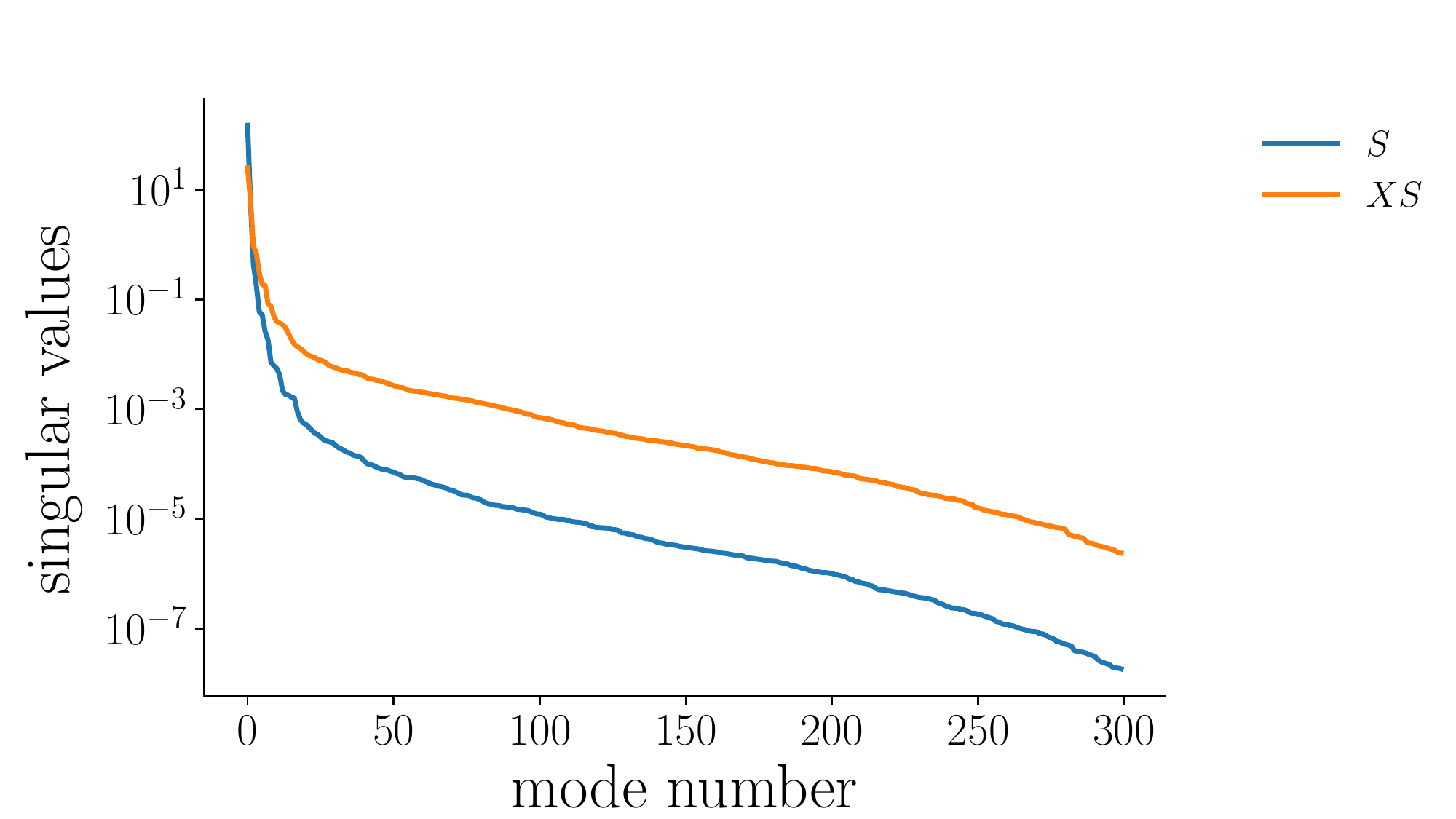} & \includegraphics[width=0.45\textwidth]{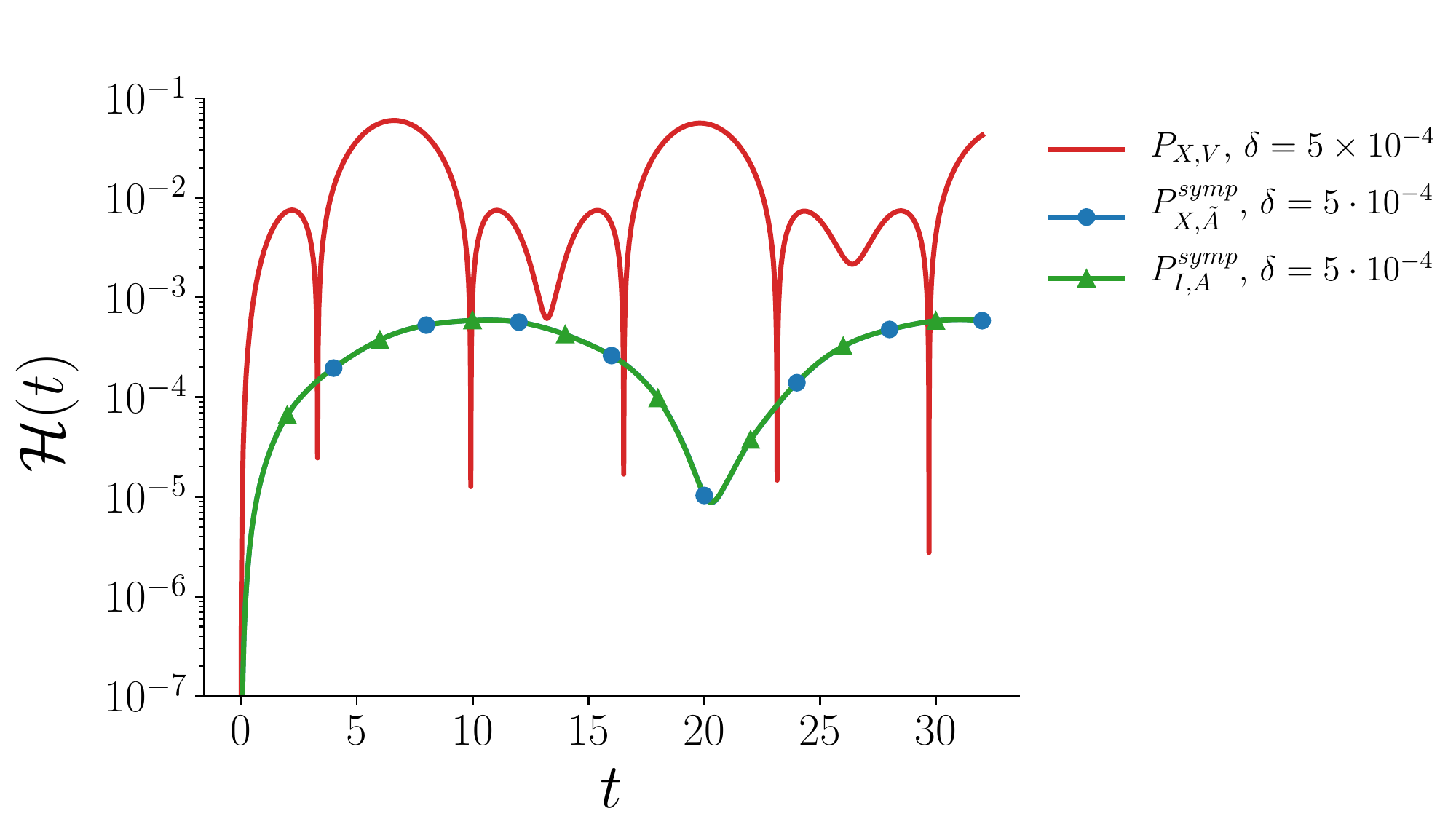} \\
(a) & (b) \\
\includegraphics[width=0.45\textwidth]{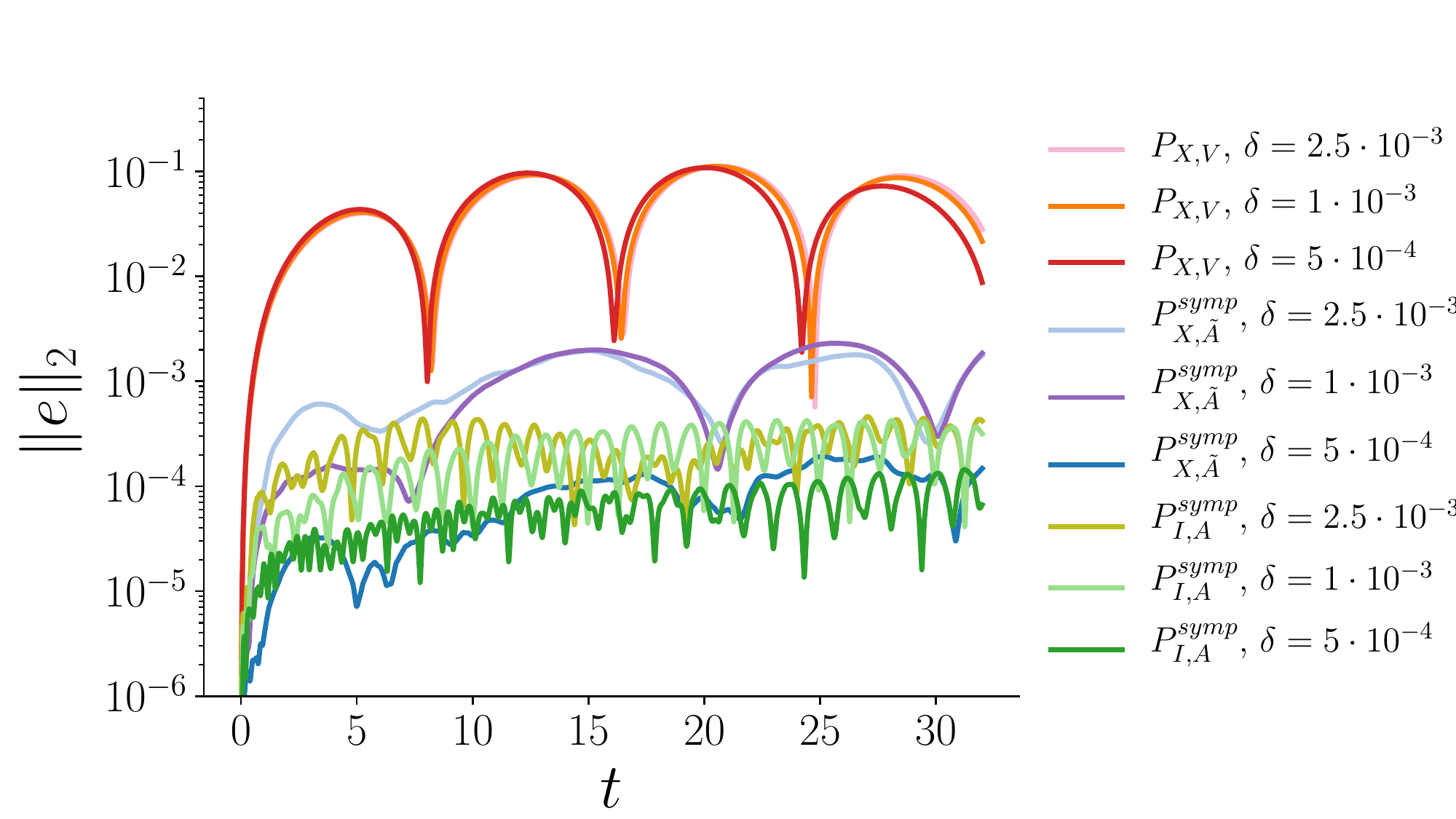} & \includegraphics[width=0.45\textwidth]{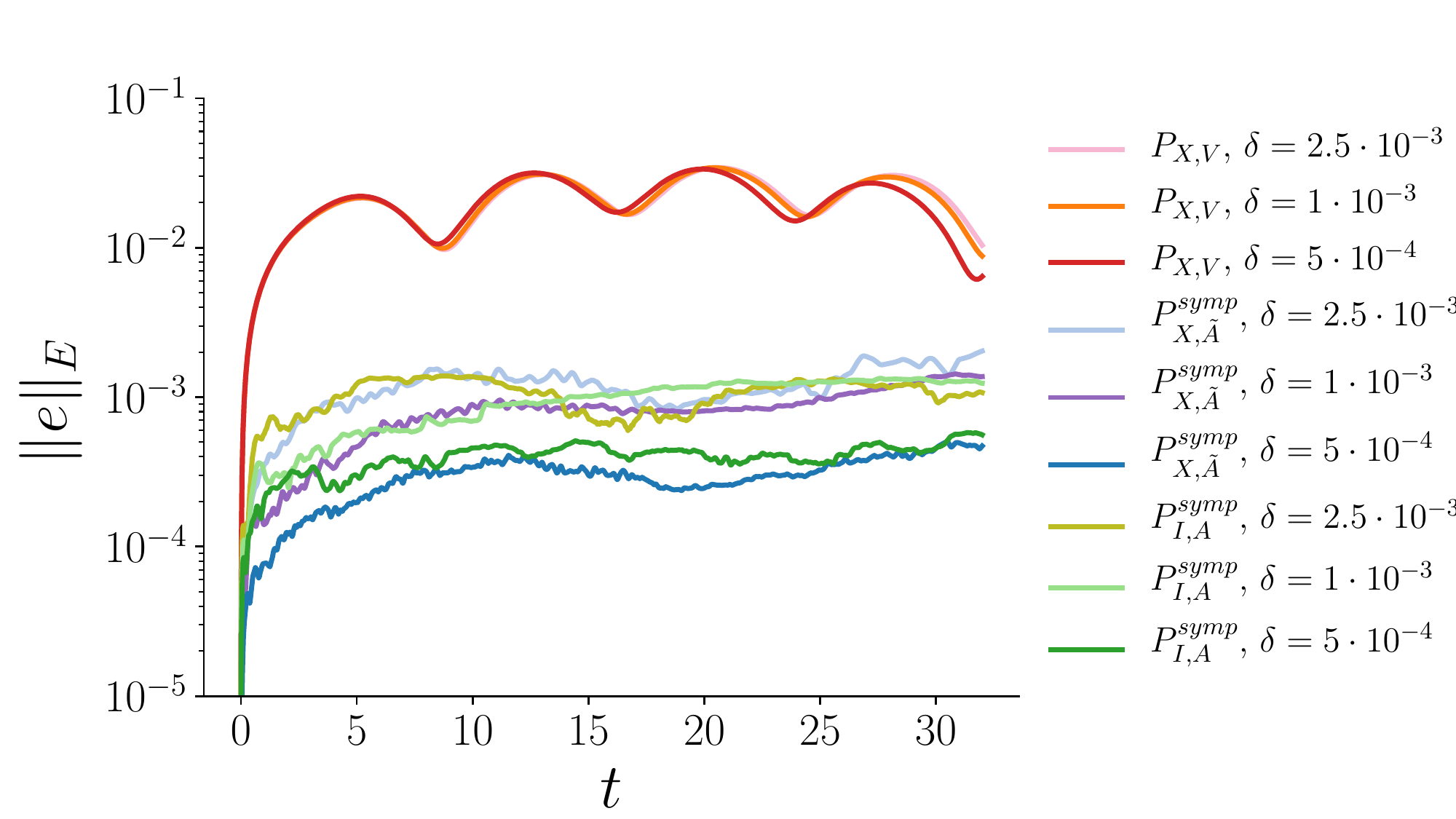} \\
(c) & (d) \\
\end{tabular}
\caption{Numerical results related to the beam with cavity. (a) the decay of the singular values. (b) conservation of the Hamiltonian. (c) error with respect to the 2-norm. (d) error with respect to the energy norm.}
\end{figure}

\Cref{fig:1.1}(a) shows the decay of the singular values for the snapshot matrix $S$ and $XS$. The divergence of the two curves indicates that to obtain the same accuracy in the reduced system, the basis constructed from $S$ and $XS$ would have different sizes.
Projection operators $P_{X,A}$, $P_{I,A}^{\text{symp}}$ and $P_{X,\tilde A}^{\text{symp}}$ are constructed according to the \cref{sec:mor.1,sec:mor.2,sec:mor.3}. The truncation error is set to $\delta = 2.5\times 10^{-3}$, 
$\delta = 1\times 10^{-3}$ and $\delta = 5\times 10^{-4}$ in \cref{alg:1,alg:2}

The 2-norm error and the error in the energy norm are presented in \Cref{fig:1.1}(c) and \Cref{fig:1.1}(d), respectively. We notice that although the non-symplectic method is bounded, it contains larger error compared to the symplectic methods. Moreover, we notice that the error generated by the symplectic methods is consistently reduced under basis enrichment. It is observed that in the energy norm, the projection $P_{X,\tilde A}^{\text{symp}}$ provides a more accurate solution (compare to \Cref{fig:1}). This is because on a nonuniform mesh, the weight matrix $X$ associates higher weights to the elements that are subject to larger error. Therefore, we expect the reduced system constructed with the projection $P_{X,\tilde A}^{\text{symp}}$ to outperform the one constructed with $P_{I,A}^{\text{symp}}$ on a highly nonuniform mesh.

\Cref{fig:1.1}(b) shows the error in the Hamiltonian. Comparing to \Cref{fig:1}, we notice that the energy norm helps with the boundedness of the non-symplectic method. However, the symplectic methods preserves the Hamiltonian at a higher accuracy

\subsection{The sine-Gordon equation} \label{sec:res.2}
The sine-Gordon equation arises in differential geometry and quantum physics \cite{Misumi2015}, as a nonlinear generalization of the linear wave equation of the form
\begin{equation} \label{eq:res.12}
\left\{
\begin{aligned}
	u_{t}(t,x) &= v, \quad x\in \Gamma,\\
	v_t(t,x) &= u_{xx} - \sin(u), \\
	u(t,0) &= 0, \\
	u(t,l) &= 2\pi.
\end{aligned}
\right.
\end{equation}
Here $\Gamma = [0,l]$ is a line segment and $u,v: \Gamma \to \mathbb R$ are scalar functions. The Hamiltonian associated with (\ref{eq:res.12}) is
\begin{equation} \label{eq:res.13}
	H(q,p) = \int_{\Gamma} \frac 1 2 v^2 + \frac 1 2 u_x^2 + 1 - \cos(u) \ dx.
\end{equation}
One can verify that $u_{t} = \delta_v H$ and $v_{t} = - \delta_u H$, where $\delta_v,\delta_u$ are standard variational derivatives. The sine-Gordon equation admits the soliton solution
\begin{equation} \label{eq:res.14}
	u(t,x) = 4 \text{arctan}\left( \exp \left( \pm \frac{x - x_0 - ct}{\sqrt{1-c^2}} \right) \right),
\end{equation}
where $x_0 \in \Gamma$ and the plus and minus signs correspond to the \emph{kink} and the \emph{anti-kink} solutions, respectively. Here $c$, $|c|<1$, is the arbitrary wave speed. We discretize the segment into $n$ equi-distant grid point $x_i = i\Delta x$, $i=1,\dots,n$. Furthermore, we use standard finite-differences schemes to discretize (\ref{eq:res.12}) and obtain
\begin{equation} \label{eq:res.15}
	\dot z = \mathbb J_{2n} L z + \mathbb J_{2n} g(z) + \mathbb J_{2n} c_b.
\end{equation}
Here $z = (q^T,p^T)^T$, $q(t) = (u(t,x_1),\dots,u(t,x_N))^T$, $p(t) = (v(t,x_1),\dots,v(t,x_N))^T$, $c_b$ is the term corresponding to the boundary conditions and
\begin{equation} \label{eq:res.16}
	L = 
	\begin{pmatrix}
		D_x^TD_x & 0_N \\
		0_N & I_n
	\end{pmatrix}, 
	\quad
	g(z) = 
	\begin{pmatrix}
	\sin(q) \\
	\vec 0
	\end{pmatrix},
\end{equation}
where $D_x$ is the standard matrix differentiation operator. We may take $X = L$ as the weight matrix associated to (\ref{eq:res.15}). The discrete Hamiltonian, takes the form
\begin{equation} \label{eq:res.17}
	H_{\Delta x} = \Delta x \cdot \frac 1 2 \| p \|^2_2 + \Delta x \cdot \| D_x q \|^2_2 + \sum_{i=1}^{n} \Delta x \cdot ( 1 - \cos(q_i) ).
\end{equation}
The system parameters are given as
\vspace{0.5cm}
\begin{center}
\begin{tabular}{|l|l|}
\hline
Domain length & $l = 50$ \\
No. grid points & $n = 500$ \\
Time step-size & $\Delta t = 0.01$ \\
Wave speed & $c=0.2$ \\
\hline
\end{tabular}
\end{center}
\vspace{0.5cm}
The midpoint scheme (\ref{eq:hamil.7}) is used to integrate (\ref{eq:res.12}) in time and generate the snapshot matrix $S$. Similar to the previous subsection, projection operators $P_{X,V}$, $P^{\text{symp}}_{I,A}$ and $P^{\text{symp}}_{X,\tilde A}$ are used to construct a reduced system. To accelerate the evaluation of the nonlinear term, the symplectic methods discussed in \cref{sec:mor.1,sec:mor.2} are coupled with the projection operators $P^{\text{symp}}_{I,A}$ and $P^{\text{symp}}_{X,A}$, respectively. Furthermore, the DEIM approximation is used for the efficient evaluation of the reduced system, obtained by the projection $P_{X,V}$. The midpoint rule is also used to integrate the reduced systems in time. \Cref{fig:2} shows the numerical results corresponding to the reduced models without approximating the nonlinearity, while the results corresponding to the accelerated evaluation of the nonlinear term are presented in \Cref{fig:3}.

\begin{figure} \label{fig:2}
\begin{tabular}{cc}
\includegraphics[width=0.45\textwidth]{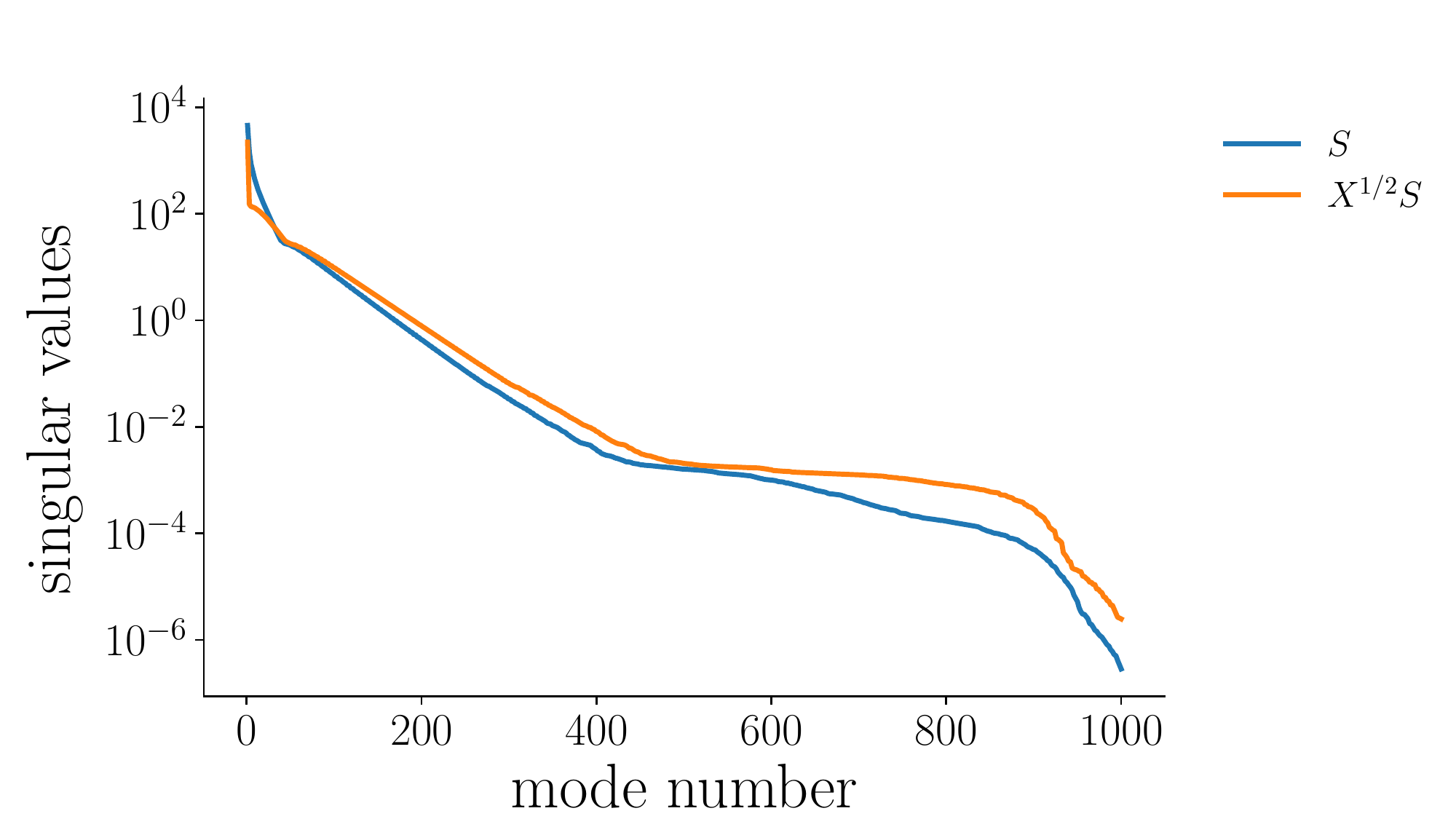} & \includegraphics[width=0.45\textwidth]{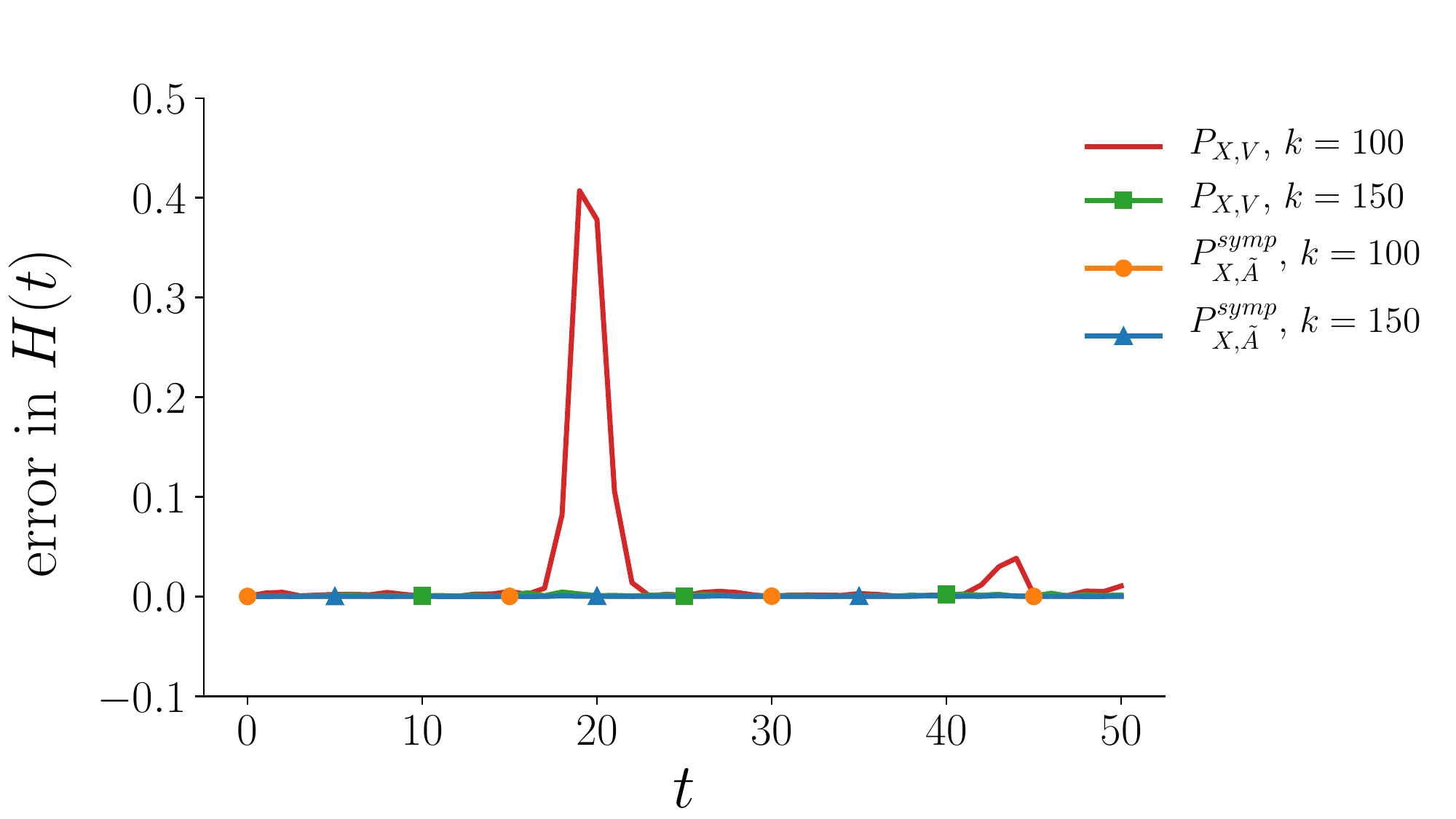} \\
(a) & (b) \\
\includegraphics[width=0.45\textwidth]{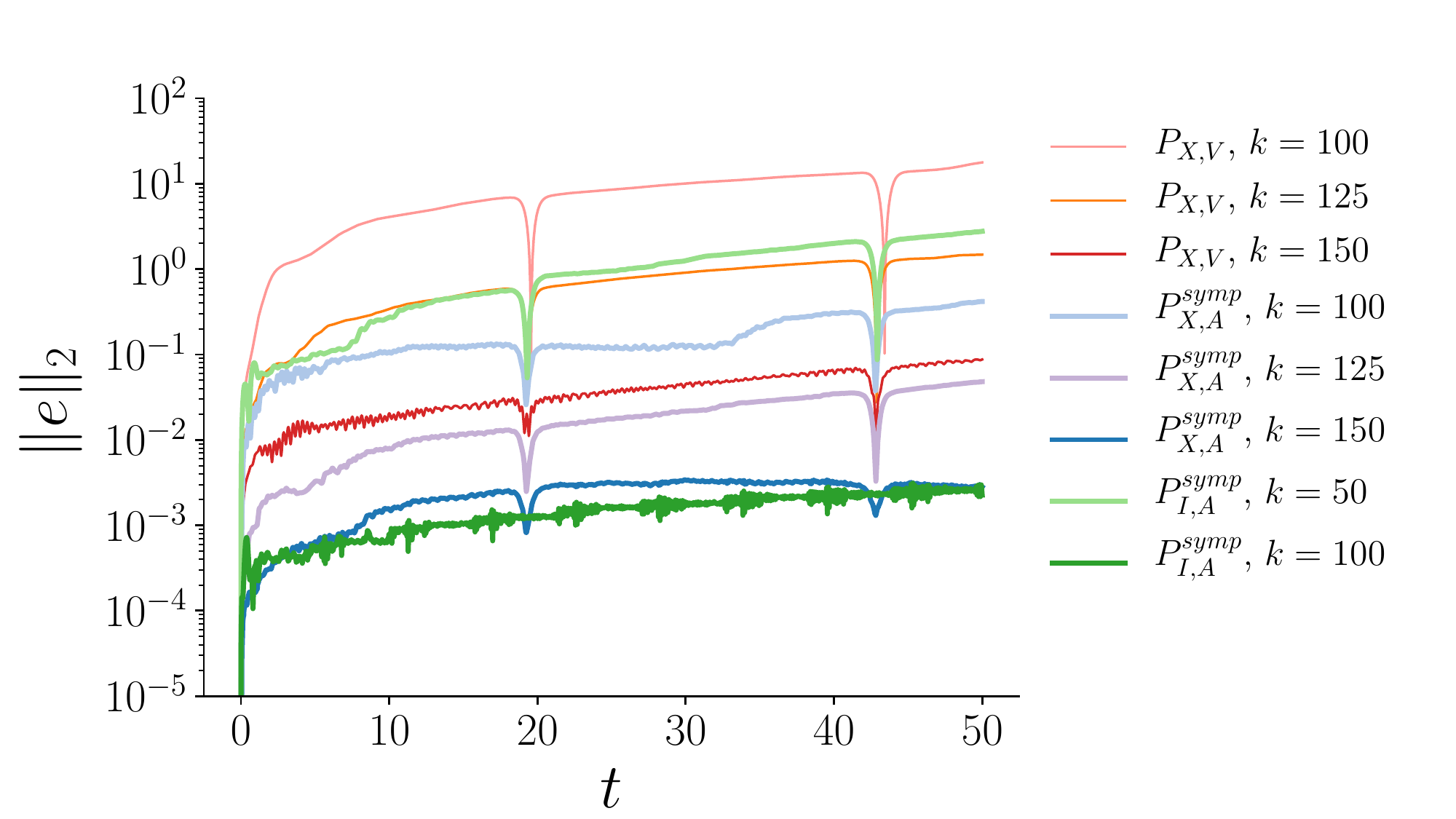} & \includegraphics[width=0.45\textwidth]{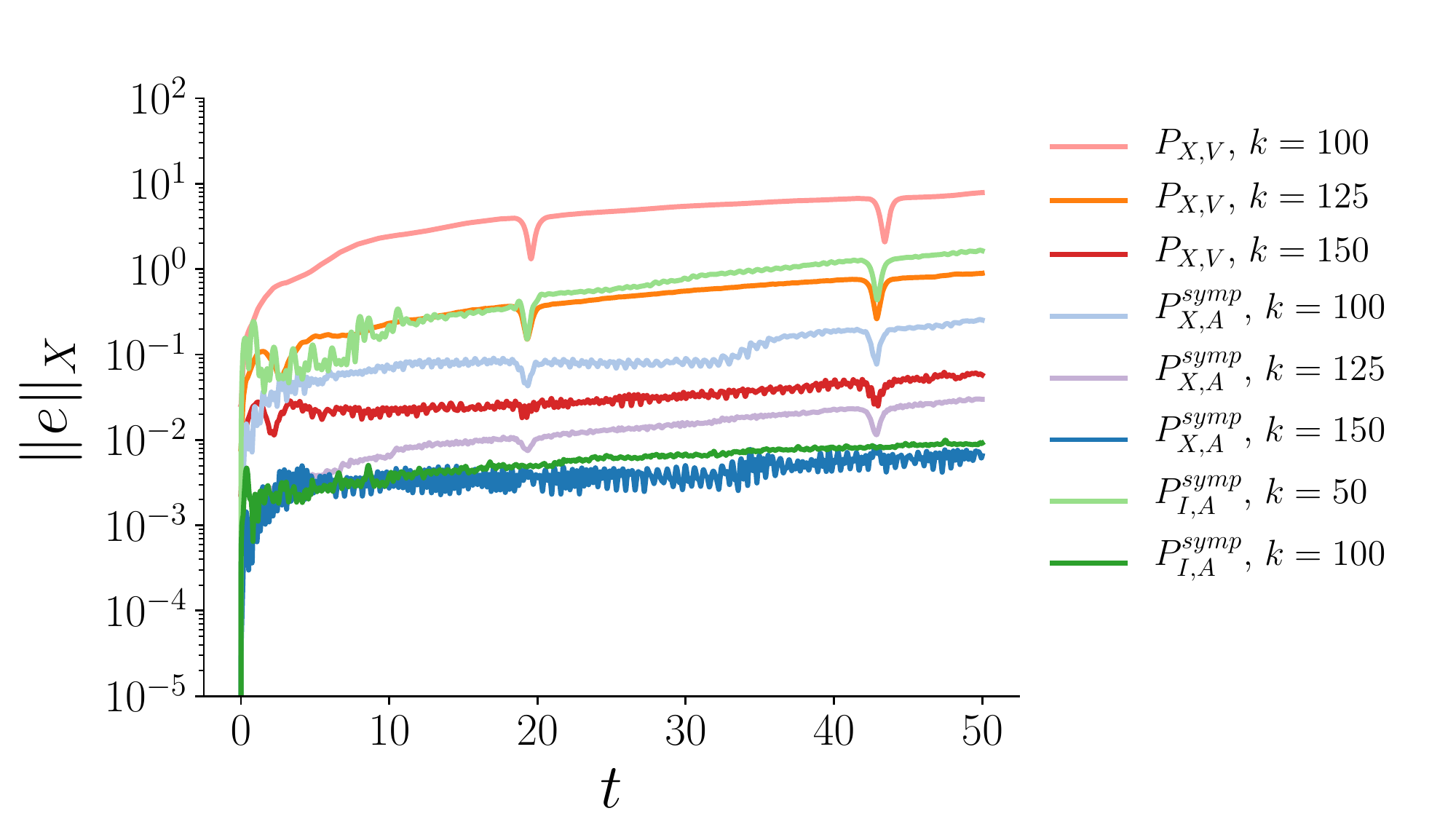} \\
(c) & (d) \\
\end{tabular}
\caption{Numerical results related to the sine-Gordon equation. (a) the decay of the singular values. (b) error in the Hamiltonian. (c) error with respect to the 2-norm. (d) error with respect to the energy norm.}
\end{figure}

\Cref{fig:2}(a) shows the decay of the singular values of matrices $S$ and $XS$. As in the previous section, we observe a saturation in the decay of the singular values of $XS$ compared to the singular values of $S$. This indicates that the reduced basis, based on a weighted inner product, should be chosen to be larger to provide an accuracy similar to based on the Euclidean inner product. Put differently, unweighted reduced bases, when compared to the weighted ones, may be highly inaccurate in reproducing underlying physical properties of the system.

\Cref{fig:2}(b) displays the error in the Hamiltonian. It is observed again that the symplectic approaches conserve the Hamiltonian. However, the classical approaches do not necessarily conserve the Hamiltonian. We point out that using the projection operator $P_{X,V}$ ensures the boundedness of the Hamiltonian. The contrary is observed when we apply the POD with respect to the Euclidean inner-product, i.e. applying the projection operator $P_{I,V}$. This can be seen in the results presented in \cite{doi:10.1137/140978922}, where the unboundedness of the Hamiltonian is observed when $P_{I,V}$ is applied to the sine-Gordon equation. Nevertheless, only the symplectic model reduction consistently preserves the Hamiltonian.

\Cref{fig:2}(c) shows the error with respect to the Euclidean inner-product between the solution of the projected systems and the original system. The behavior of the solution is investigated for $k=100$, $k=125$ and $k=150$. We observe that all systems which are projected with respect to the $X$-norm are bounded. As the results in \cite{doi:10.1137/140978922} suggest, the Euclidean inner-product does not necessarily yield a bounded reduced system. Moreover, we notice that the symplectic projection $P^{\text{symp}}_{X,\tilde A}$ results in a substantially more accurate reduced system compared to the reduced system yielded from $P_{X,V}$. This is because the overall behavior of the original system is translated correctly to the reduced system constructed with the symplectic projection.

The error with respect to the $X$-norm between the solution of the original system and the projected systems is presented in \Cref{fig:2}(d). We see that the behavior of the $X$-norm error is similar to the Euclidean norm, however the growth of the error is slower for methods based on a weighted inner product. Note that the connection between the error in the Euclidean norm and the $X$-norm is problem and discretization dependent. We also observed that symplectic methods are substantially more accurate.

\begin{figure} \label{fig:3}
\begin{tabular}{cc}
\includegraphics[width=0.45\textwidth]{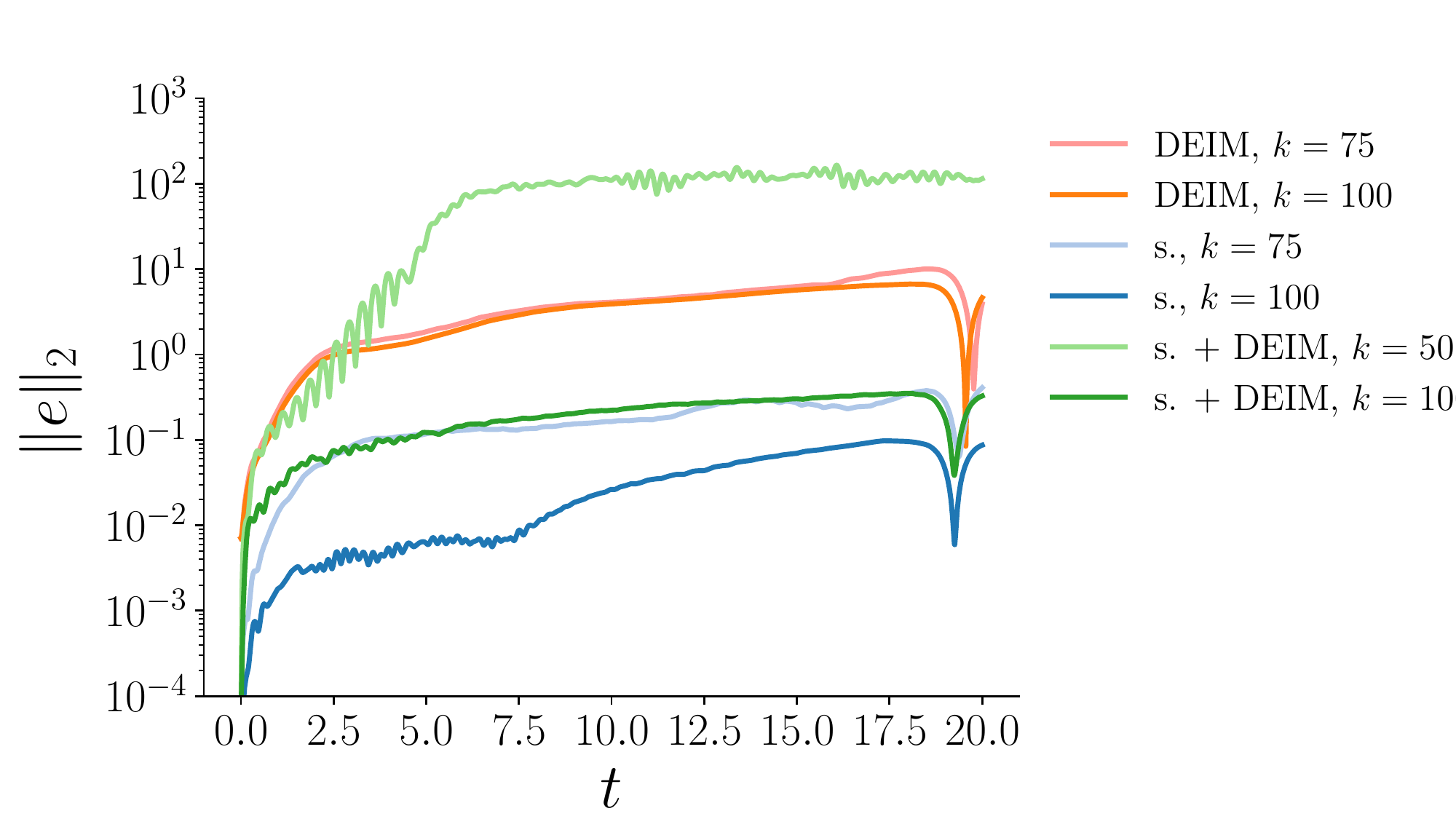} & \includegraphics[width=0.45\textwidth]{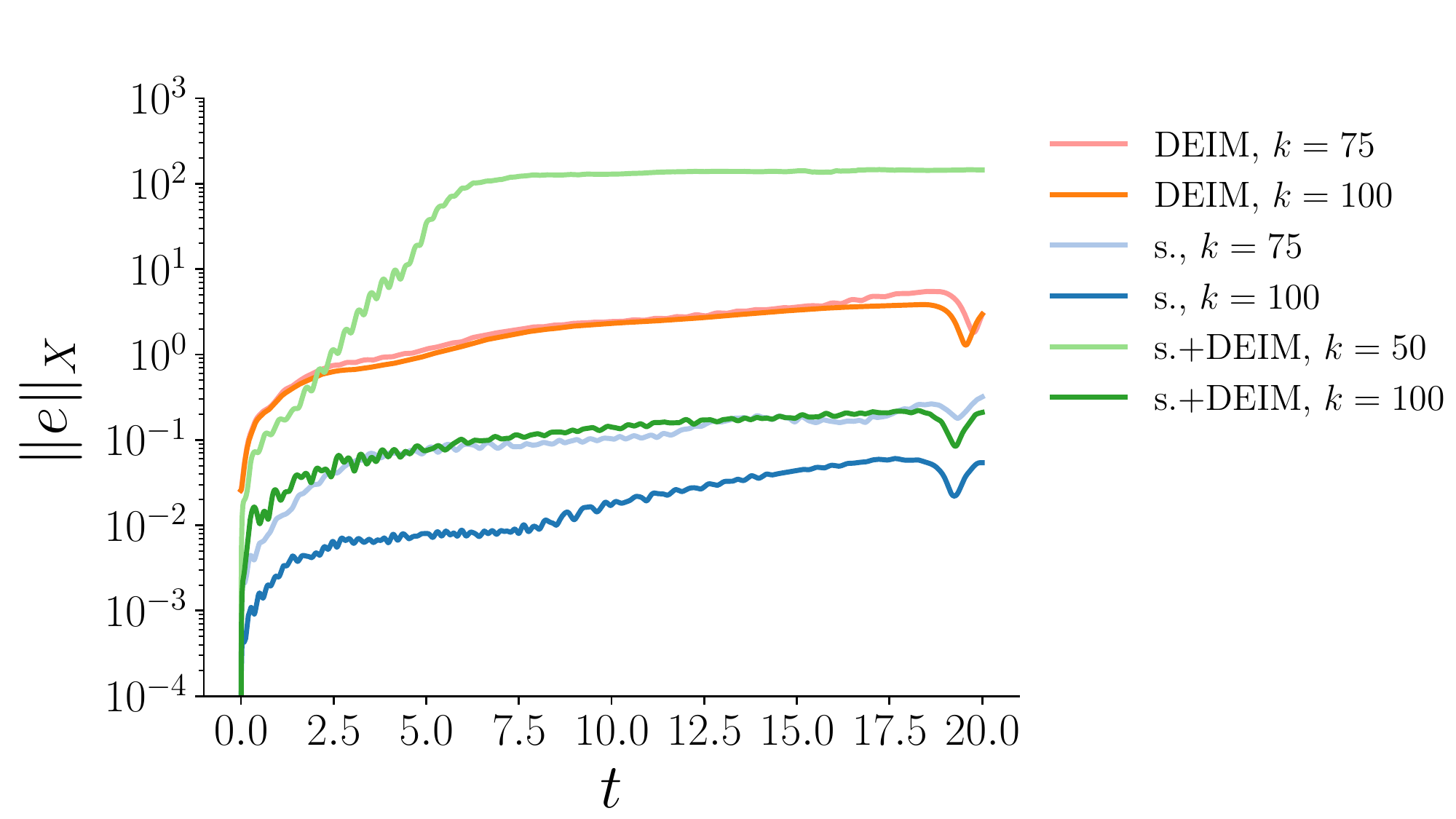} \\
(a) & (b) \\
\multicolumn{2}{c}{
\includegraphics[width=0.45\textwidth]{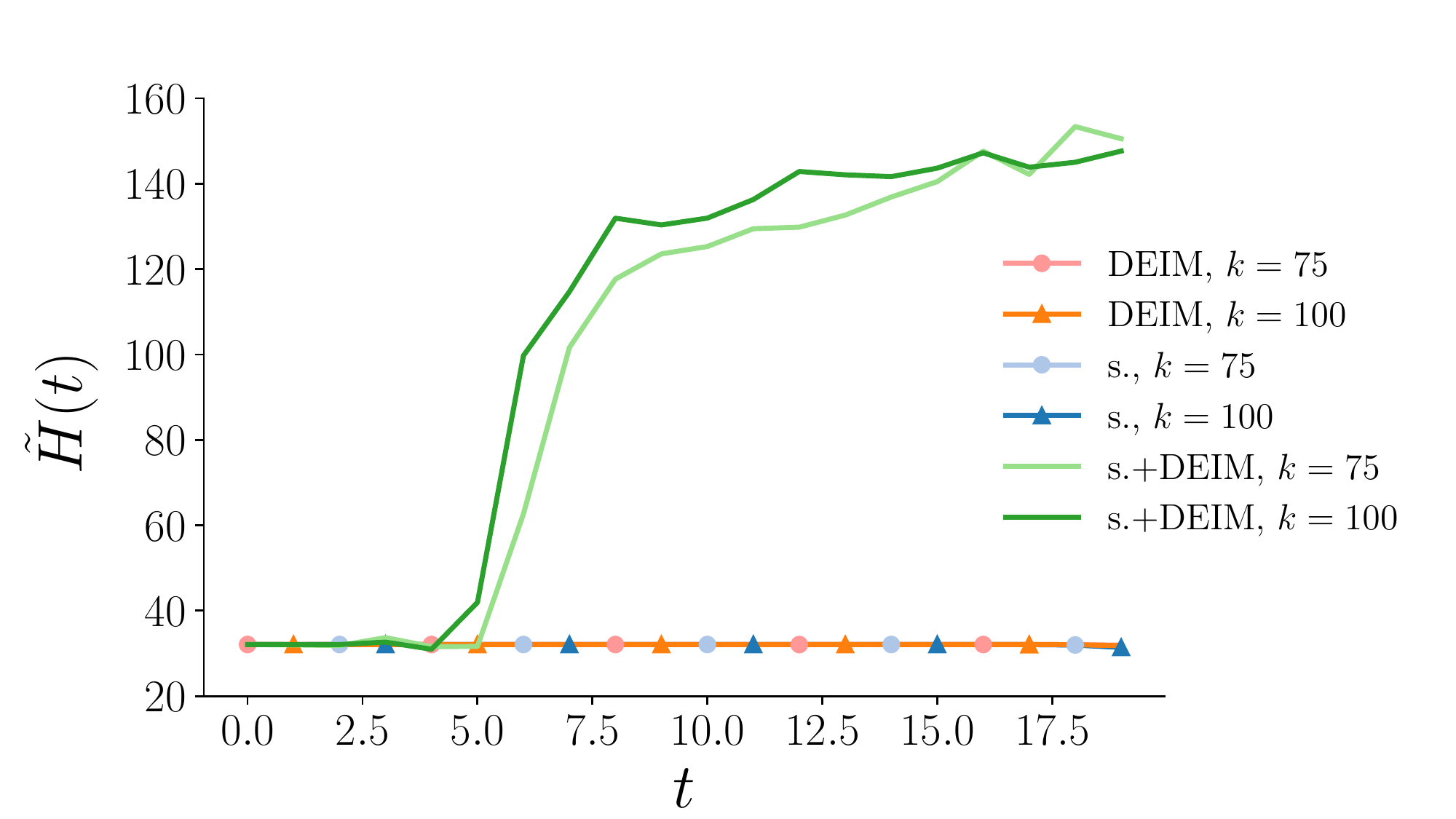}
} \\
\multicolumn{2}{c}{(c)} \\
\end{tabular}
\caption{Numerical results related to the sine-Gordon equation with efficient evaluation of the nonlinear terms. Here, ``DEIM'' indicates classical model reduction with the DEIM, ``s.+DEIM'' indicates symplectic model reduction with the DEIM and ``s.'' indicates symplectic model reduction with symplectic treatment of the nonlinear term. (a) error with respect to the Euclidean norm. (b) error with respect to the $X$-norm. (c) error in the Hamiltonian. }
\end{figure}

\Cref{fig:3} shows the performance of the different model reduction methods, when an efficient method is adopted in evaluating the nonlinear term in (\ref{eq:res.15}). This figure compares the symplectic approaches against non-symplectic methods. For all simulations, the size of the reduced basis for (\ref{eq:res.15}) is chosen to be $k=100$. The size of the basis of the nonlinear term is then taken as $k_n=75$ and $k_n=100$. For symplectic methods, a basis for the nonlinear term is constructed according to \Cref{alg:3}, whereas for non-symplectic methods, the DEIM is applied. Note that for symplectic methods, the basis for the nonlinear term is added to the symplectic basis $A$. This means that the size of the reduced system is larger compared to the classical approach.

\Cref{fig:3}(a) and \Cref{fig:3}(b) show the error with respect to the Euclidean norm and the $X$-norm between the solution of the projected systems compared to the solution of the original system, respectively. We observe that all solutions are bounded and the behavior of the error in the Euclidean norm and the $X$-norm is similar. We observe that enriching the DEIM basis does not increase the overall accuracy of the system projected using $P_{X,V}$. Furthermore, applying the DEIM to a symplectic reduced system also destroys the symplectic nature of the reduced system, as suggested in \cref{sec:normmor.3}. Therefore, it is essential to adopt a symplectic approach to reduce the complexity of the evaluation of the nonlinear terms. We observe that the symplectic method presented in \cref{sec:normmor.3} provides not only an accurate approximation of the nonlinear term, but also preserves the symplectic structure of the reduced system. Moreover, enriching such a basis consistently increases the accuracy of the solution, as suggested in \Cref{fig:3}(a) and \Cref{fig:3}(b).

\Cref{fig:3}(b) shows the conservation of the Hamiltonian for different methods. It is again visible that applying the DEIM to a symplectic reduced system destroys the Hamiltonian structure, therefore the Hamiltonian is not preserved.

\section{Conclusion} \label{sec:conc}
We present a model reduction routine that combines the classic model reduction method, defined with respect to a weighted inner product, with symplectic model reduction. This allows the reduced system to be defined with respect to the norms and inner-products that are natural to the problem and most suitable for the method of discretization. Furthermore, the symplectic nature of the reduced system preserves the Hamiltonian structure of the original system, which results in robustness and enhanced stability in the reduced system.

We demonstrate that including the weighted inner-product in the symplectic model reduction can be viewed as a natural extension of the unweighted symplectic method. Therefore, the stability preserving properties of the symplectic method generalize naturally to the new method.

Numerical results suggest that classic model reduction methods with respect to a weighted inner product can help with the boundedness of the system. However, only the symplectic treatment can consistently increase the accuracy of the reduced system. This is consistent with the fact the symplectic methods preserve the Hamiltonian structure.

We also show that to accelerate the evaluation of the nonlinear terms, adopting a symplectic approach is essential. This allows an accurate reduced model that is consistently enhanced when the basis for the nonlinear term is enriched.

Hence, the symplectic model-reduction with respect to a weighted inner product can provide an accurate and robust reduced system that allows the use of the norms and inner products most appropriate to the problem.

\section*{Acknowledgments} We would like to show our sincere appreciation to Dr. Claudia Maria Colciago for the several brainstorming meetings which helped with the development of the main parts of this article. We would also like to thank Prof. Karen Willcox for hosting Babak Maboudi Afkham at MIT during the composition of this paper.

\bibliographystyle{siamplain}
\bibliography{references}
\end{document}